\documentclass[12pt]{amsart}
\usepackage{amsmath,amsthm,amscd,euscript,amssymb,comment,epsfig,
  hyperref,mathdots}
\usepackage{graphicx}
\usepackage{epstopdf}
\usepackage{color}
\usepackage[dvipsnames]{xcolor}
\usepackage{colortbl}
\numberwithin{figure}{section}
\numberwithin{table}{section}
\newtheorem{theorem}{Theorem}[section]

\newtheorem{lemma}[theorem]{Lemma}
\newtheorem{prop}[theorem]{Proposition}

\theoremstyle{definition}
\newtheorem{definition}[theorem]{Definition}
\newtheorem{example}[theorem]{Example}

\newtheorem{cor}[theorem]{Corollary}

\theoremstyle{remark}
\newtheorem{remark}[theorem]{Remark}

\numberwithin{equation}{section}
\newfont{\tap}{tap scaled 650}

\def \H{{\mathbb H}}
\def \E{{\mathbb E}}

\def \Q{{\mathbb Q}}
\def \C{{\mathbb C}}
\def \R{{\mathbb R}}

\def \Z{{\mathbb Z}}

\def \KK{\widehat{K}}
\def \QQ{\widehat{\mathbb Q}}
\def \CC{\widehat{\mathbb C}}
\def \RR{\widehat{\mathbb R}}
\def \ZZ{\widehat{\mathbb Z}}
\def \SL{SL}
\def \PGL{ PGL}

\def \[{[ }
\def \]{] }

\def \T{{\mathcal T}}

\def\x{\mathbf{x}}

\renewcommand{\bar}{\overline}

\definecolor{dgreen}{rgb}{0,0.5,0}
\definecolor{dred}{rgb}{0.5,0,0}

\def \tsin{{\angle}_\mathcal T}

\begin{document}
\thispagestyle{empty}
\title[$3$D Farey graph, lambda lengths and $SL_2$-tilings]{$3$D Farey graph, lambda lengths and $SL_2$-tilings}
\author[A.~Felikson]{Anna Felikson}
\address{Department of Mathematical Sciences, Durham University, Upper Mountjoy Campus, Stockton Road, Durham, DH1 3LE, UK}
\email{anna.felikson@durham.ac.uk} 
\thanks{Research was supported in part by the Leverhulme Trust research grant RPG-2019-153 (PT) and NSF grant DMS-2054255 (KS)}
\author[O.~Karpenkov]{Oleg Karpenkov}
\address{Department of Mathematical Sciences, University of Liverpool,
Mathematical Sciences Building, Liverpool L69 7ZL, UK}
\email{karpenk@liverpool.ac.uk}

\author[K.~Serhiyenko]{Khrystyna Serhiyenko}
\address{University of Kentucky, Lexington, Department of Mathematics, 951 Patterson Office Tower, Lexington, KY 40506-0027, USA}
\email{khrystyna.serhiyenko@uky.edu}

\author[P.~Tumarkin]{Pavel Tumarkin}
\address{Department of Mathematical Sciences, Durham University,  Upper Mountjoy Campus, Stockton Road, Durham, DH1 3LE, UK}
\email{pavel.tumarkin@durham.ac.uk}

\begin{abstract}
We explore a three-dimensional counterpart of the Farey tessellation and its relations to Penner's lambda lengths and $SL_2$-tilings. In particular, we prove a three-dimensional version of Ptolemy relation, and generalise results of Short~\cite{Sh} to classify tame  $SL_2$-tilings over Eisenstein integers in terms of pairs of paths in the $3$D Farey graph. 

\end{abstract}

\maketitle
\setcounter{tocdepth}{1}
\tableofcontents

\section{Introduction and main results}

We study geometric aspects of Farey graph over Eisenstein integers and its realisation in the hyperbolic three-dimensional space as the $1$-skeleton of the union of the symmetry planes (including points at the absolute) 
of the reflection group of the regular ideal hyperbolic tetrahedron.
Our first main goal is to generalise relations between Penner's $\lambda$-lengths and $SL_2(\Z)$-tilings and to prove 
a three-dimensional version of Ptolemy relation. Secondly, we classify tame  $SL_2$-tilings over Eisenstein integers 
in terms of pairs of paths in the $3$D Farey graph.

The classical notion of Farey graph, together with its close relatives such as circle packings, continued fractions, Conway-Coxeter friezes and $SL_2$-tilings, is a subject of large and ever growing literature. 
Overviews of different aspects of the theory can be found in~\cite{Hat,MGO,S}.  Moreover, Farey graph also appears in the study of discrete group of symmetries of the hyperbolic plane $\H^2$, which has been studied exhaustively. It is then interesting to ask about which features of this theory can be generalised or extended to higher dimensions.

There are many natural generalisations of the above mentioned classical notions, based on substitution of $\Z$ with various rings.
In 1887 in his fundamental work~\cite{H}, A.~Hurwitz initiated a systematic study of the continued fractions 
over $\C$ and over various subrings in $\C$ (see also \cite{H2}).
In 1951, 
Cassels, Ledermann and Mahler made the first steps in Farey graphs for Gaussian and Eisenstein numbers in~\cite{CLM}. 
A little later in his paper~\cite{Sch},  Schmidt introduced a counterpart of the Farey graph for all imaginary quadratic fields $K=\Q(\sqrt{-d})$, 
where $d$ is square-free, see also paper~\cite{V} by Vulakh for further generalisations (recall that $d=1,3$ 
correspond to the Gaussian and Eisenstein numbers considered in~\cite{CLM}).  We would also like to mention an essentially 
different three-dimensional approach to Farey graph developed by Beaver and Garrity~\cite{BG}, based on multidimensional Farey addition that does not appear in the theory of complex continued fractions.
Furthermore, various other objects related to Farey graphs have also been extended beyond the classical setting. 
In particular, a three-dimensional analogue of Ford circles  for  $d=1$ was introduced already by Ford~\cite{F}; for other fields see e.g.~\cite{R,N}.
In~\cite{St}, Stange studied various circle packings arising from Bianchi groups $PSL_2(\mathcal O_K)$, 
where $\mathcal O_K$ is the ring of integers of an imaginary quadratic field $K$. 
Coxeter~\cite{C2} considered examples of friezes with quadratic irrational entries,
Holm and Jorgensen~\cite{HJ} used $n$-angulations of polygons to classify friezes with quiddity 
row consisting of positive integer multiples of $2\cos (\pi/p)$.

\bigskip

In this paper, we consider a $3$-dimensional analogue of the Farey graph arising from a tessellation of hyperbolic space $\H^3$ by regular hyperbolic ideal simplices (used in place of a tessellation of $\H^2$ by ideal triangles). 
We call it
the {\it tetrahedral graph} $\mathcal T$. The graph $\mathcal T$ inherits many good properties of the classical Farey graph $\mathcal F$ (see Section~\ref{3D Farey graph: tetrahedral graph} for details and essential definitions). In particular, the vertices of $\mathcal T$ are precisely points of  $\QQ(\sigma)=\Q(\sigma)\cup \{\infty \}$, where $\sigma= e^{i\pi/3}=\frac12+i\frac{\sqrt{3}}{2}$, the group  of symmetries of $\mathcal T$ is the Bianchi group $Bi(3)$, and the edges of $\mathcal T$ can be described, similarly to the ones of the Farey graph, via determinants: two irreducible fractions $p/q$ and $r/s\in \QQ(\sigma)$ are joined by an edge if and only if $|ps-rq|=1$ (see Section~\ref{sec-det}).
Furthermore, as for the Farey graph, faces of $\mathcal T$ can be described via Farey addition (see Section~\ref{sec-faces}).

Another property inherited by the tetrahedral graph is the relation with {\em $\lambda$-lengths}. Given two points $x,y\in \partial \H^d$ and a choice of horospheres $h_x,h_y$ centred at $x$ and $y$, Penner~\cite{P} introduced the notion of $\lambda$-length $\lambda_{xy}$ between $x$ and $y$ as $\lambda_{xy}=e^{d/2}$, where $d$ is the signed distance between $h_x$ and $h_y$. Penner also showed that for an ideal quadrilateral $xyzt$, the corresponding $\lambda$-lengths satisfy Ptolemy relation
$$ \lambda_{xz}\lambda_{yt}= \lambda_{xy}\lambda_{zt}+ \lambda_{yz}\lambda_{xt}.$$

Given two irreducible fractions $p/q, r/s\in \QQ(\sigma)$, we can also define the {\em det-length}  $l(p/q,r/s)$ as the absolute value of the determinant $l(p/q,r/s)=|ps-rq|$. We then choose a distinguished set of horospheres at points of  $\QQ(\sigma)$ (see Section~\ref{sigma-pt}; the horospheres are represented by {\em Ford spheres}~\cite{R}) and show that $\lambda$-lengths computed with respect to these horospheres coincide with det-lengths.
\setcounter{section}{4}
\setcounter{theorem}{11}
\begin{theorem}
\label{l=l}
  Let $X,Y\in \QQ(\sigma)$ be two  irreducible fractions. Let the standard horosphere be chosen at every point of $\QQ(\sigma)$.
Then $\lambda_{XY}=l_{XY}$.
\end{theorem}

To prove Theorem~\ref{lambda=l}, we first show that $\lambda$-lengths between vertices of $\mathcal T$ satisfy an analogue of the Ptolemy relation:

\setcounter{theorem}{6}
\begin{theorem}
\label{thm-fund}
  Let $A_1A_2A_3A_4$ be a fundamental tetrahedron of $\mathcal T$ with vertices in  $\QQ(\sigma)$,  choose any $X\in  \QQ(\sigma)$ distinct from $A_i$.
Let $\lambda_i=\lambda_{XA_i}$ be the $\lambda$-length of $XA_i$, $i=1,\dots,4$.
Then
$$
\sum\limits_{i=1}^4 \lambda_i^4=\sum\limits_{1\le i<j\le 4} \lambda_i^2\lambda_j^2.
$$
\end{theorem}

We prove Theorem~\ref{1} in two ways, i.e. by a direct computation and as a corollary of the Soddy-Gosset theorem relating the radia of five mutually tangent spheres in $\R^3$. We then show (in Theorem~\ref{2}) that det-lengths satisfy the same Ptolemy relation, which eventually implies Theorem~\ref{lambda=l}.

We then prove a $3$-dimensional counterpart of Ptolemy relation which can be applied to any $5$ points in $\CC$:
\setcounter{theorem}{16}
\begin{theorem}
Let $X_1,\dots,X_5\in \CC=\partial \H^3$ be $5$ distinct points. Suppose that there are horospheres chosen in these points. Let $\lambda_{ij}=\lambda_{X_iX_j}$.  Then
$$ \sum\limits_{i<j} \lambda_{ij}^4 \lambda_{kl}^2\lambda_{lm}^2\lambda_{mk}^2= \sum\limits_{\text{{\rm cycles} $(ijklm)$}} \lambda_{ij}^2\lambda_{jk}^2\lambda_{kl}^2\lambda_{lm}^2\lambda_{mi}^2,$$
where all indices $i,j,k,l,m$ are distinct.
\end{theorem}

Next, we apply $\mathcal T$ to generalise results of Short~\cite{Sh} to classify $\SL_2(\Z[\sigma])$-tilings.
A {\em path} $(v_i)$ in $\mathcal T$ is a (bi-infinite) sequence of vertices of $\mathcal T$ such that $v_i$ and $v_{i+1}$ are connected by an edge of $\mathcal T$. We normalise the paths by requiring that the expressions $v_i=p_i/q_i$ satisfy the condition $p_iq_{i+1}-p_{i+1}q_i=1$. We then prove the following result, which is a direct generalisation of the result of~\cite{Sh}.

\setcounter{section}{5}
\setcounter{theorem}{19}
\begin{theorem}
\label{thm-eq}
  Given two normalised paths  $v_i=p_i/q_i$ and $u_j=r_j/s_j$, the map $(u_i,v_j)\!\mapsto m_{ij}=p_is_j-q_ir_j$ provides a bijection between equivalence classes of the tame $SL_2(Z[\sigma])$-tilings and   pairs of paths in $\mathcal T$ considered up to simultaneous action of $SL_2 (Z[\sigma])$ on both paths.
\end{theorem}

\setcounter{section}{1}
\setcounter{theorem}{0}

Here, we say that two  $SL_2(Z[\sigma])$-tilings are equivalent if one is obtained from the other by multiplication of even rows by $\sigma^{k}$    and of odd rows  by  $\sigma^{-k}$, together with multiplication by  $\sigma^{l}$ (resp.  $\sigma^{-l}$) of even (resp. odd) columns.

In fact, the bijection given in Theorem~\ref{sh-det} can be refined to enumeration of individual tilings rather than equivalence classes. This is achieved by using  {\em $\mathcal T$-angle sequences} of paths we introduce in Section~\ref{sec-path}.

Given a path in  $\mathcal T$,  we construct a sequence of numbers from  $\Z[\sigma]$ called $\mathcal T$-angle sequence of the path  (also known as quiddity sequences for friezes or as itinerary in~\cite{Sh}). We then use Theorem~\ref{sh-det} to provide a geometric interpretation of the classification of $SL_2$-tilings obtained in~\cite{BR}.

In conclusion,
we note that in the classical two-dimensional setting there are also beautiful connections between the combinatorics of $SL_2$-tilings, triangulations of polygons or more generally apeirogons, cluster algebras,  as well as representation theory of quivers.  It would be interesting to see to which extent it is possible to incorporate these ideas into three-dimensional setting. One example of connections between cluster algebras and triangulated three-dimensional hyperbolic manifolds can be found in~\cite{NTY}.

\bigskip
\noindent
The paper is organised as follows.

In Section~\ref{Preliminaries}, we recall basic properties of the classical Farey graph $\mathcal F$.
In Section~\ref{3D Farey graph: tetrahedral graph}, we introduce the tetrahedral graph $\mathcal T$ and describe its properties.
 Section~\ref{Relations on lambda-lengths} is devoted to establishing various relations on $\lambda$-lengths, first between the points of $\QQ(\sigma)$ and then generally in $\H^3$.
In Section~\ref{SL_2-tilings}, we show that $SL_2$-tilings over $\Z[\sigma]$ are classified by pairs of paths in $\mathcal T$.
In Section~\ref{sec-path}, we describe paths in  $\mathcal T$ in terms of sequences of numbers from  $\Z[\sigma]$ ($\mathcal T$-angles of the paths), we then enumerate $SL_2$-tilings over $\Z[\sigma]$ by pairs of infinite sequences (together with an element of $SL_2(\Z[\sigma])$).
Finally, Section~\ref{comments} is devoted to various remarks, including the connection between   paths in $\mathcal T$ and continued fractions over $\Z[\sigma]$, discussion of some properties of $\mathcal T$-angles, and generalisation of the results to other imaginary quadratic fields. We also give a geometric argument to reprove a recent result of Cuntz and Holm~\cite{CH2} concerning friezes over algebraic numbers.

\subsection*{Acknowledgements}
The work was initiated and partially done at the Isaac Newton Institute for Mathematical Sciences, Cambridge; we are grateful to the organizers of the program ``Cluster algebras and representation theory'', and to the Institute for support and hospitality during the program; this work was supported by EPSRC grant no EP/R014604/1.
We would like to thank Arthur Baragar, Sergey Fomin, Ivan Izmestiev, Valentin Ovsienko and Ian Whitehead for helpful discussions.

\section{Preliminaries: Farey graph  and its symmetries}
\label{Preliminaries}
In this section we recall some classic notions and definitions.
In particular we recall the definitions of Farey graphs,  Farey addition, Ford cycles, and Farey tessellation of 
the hyperbolic plane. We also discuss the group of symmetries of the Farey tessellation.

\smallskip
\noindent
{\bf Notation. }
Throughout the paper we denote $\CC=\C\cup \infty$. We also use the notations $\RR$, $\QQ$, $\ZZ$, $\ZZ[\sigma]$ for the sets
$\R\cup \infty$, $\Q\cup \infty$, $\Z\cup \infty$,  and $\Z[\sigma]\cup \infty$
respectively.

\subsection{Farey graph}

\begin{definition}
The {\it Farey graph} $\mathcal F$ is an infinite graph whose vertices are $\QQ$;
two vertices $u,v\in\QQ$ with irreducible fractions $p/q$ and $r/s$ respectively are
connected by an edge if and only if $|ps-rq|=1$.
\end{definition}

It is convenient to visualise the Farey graph $\mathcal F$ by drawing it in the upper half-plane model
of the hyperbolic plane $\H^2$.
Namely, we identify $\H^2$ with $\{z\in \C \ | \ \mathrm{Im(z)}>0\}\subset \CC$; its boundary is therefore identified with $\RR$.
Then the vertices of the graph $\mathcal F$  are naturally identified with rational points of the absolute;
for each edge $uv$ in the graph $\mathcal F$ we draw a  hyperbolic line connecting $u$ and $v$ (i.e.
a semicircle with endpoints at $u$ and $v$), see Figure~\ref{farey}.

\begin{figure}[!h]
\begin{center}
 \includegraphics[width=0.99\linewidth]{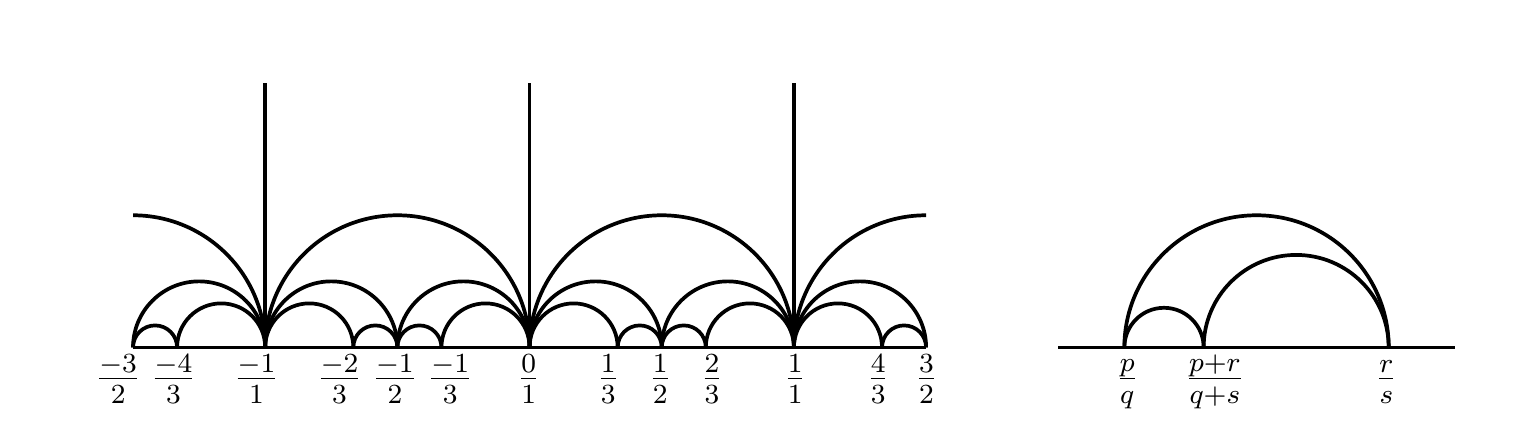}     
\caption{Farey graph}
\label{farey}
\end{center}
\end{figure}

Observe that the geodesic semicircles representing any two distinct edges of the graph do not intersect each other in $\H^2$,
however they might have a common vertex on the absolute.
Moreover, the obtained diagram is a tessellation of $\H^2$ by ideal triangles (i.e. geodesic triangles with vertices at the absolute); we call it {\it Farey tessellation}.

\subsection{Farey addition in $\QQ$ and ideal triangles in the tessellation}
\label{sec_Farey_add}
Given an edge connecting $p/q$ with $r/s$ in the Farey graph, the third vertex of the triangle lying right below this edge is given by the {\it Farey addition}, i.e. $$\frac{p}{q}\oplus\frac{r}{s}=\frac{p+r}{q+s},$$
see Figure~\ref{farey}, right.
One can check that the third vertex of the triangle lying right above the edge can be written by  $\frac{p}{q}\ominus\frac{r}{s}=\frac{p-r}{q-s}$, and that in terms of hyperbolic geometry the two triangles can be obtained from each other by applying a reflection with respect  the hyperbolic line connecting  $p/q$ with $r/s$.
Note that
$
u\ominus v=v\ominus u.
$

\subsection{Symmetry group of $\mathcal F$}
Recall that the group $PSL_2(\Z)$ naturally acts on the upper half-plane, here
for every $\begin{pmatrix} a&b\\ c&d\end{pmatrix}\in \SL_2(\Z)$
we consider the mapping $z\to \frac{az+b}{cz+d}$.
This action preserves the value  $|ps-rq|$ for every two irreducible fractions $p/q,r/s\in \QQ$.
Therefore, every element of  $PSL_2(\Z)$ takes the Farey graph to itself.
Thus, we can identify $PSL_2(\Z)$ with a subgroup of symmetries  of $\mathcal F$.

\vspace{2mm}

In fact, $PSL_2(\Z)$ is the group of orientation-preserving symmetries of $\mathcal F$. This group is an index 2 subgroup of the  group of all symmetries of $\mathcal F$, the latter  is generated by reflections with respect to the the sides of the ideal triangle with vertices $0,1,\infty$.
The group $PSL_2 (\Z)$ acts transitively on the vertices of $\mathcal F$, on the edges of $\mathcal F$, and on the triangles in the tessellation.

\subsection{Ford circles}
Consider the Euclidean circle of (Euclidean) radius $1/2$ tangent to the real axis at the point $0$ (in terms of hyperbolic geometry it represents a horocycle). The action of $PSL_2(\Z)$ takes this circle to infinitely many circles tangent to the real axis at each rational point, see Figure~\ref{ford}. All such circles are called {\it Ford circles}.
Observe that for every pair of  rational points connected by an edge in the Farey graph the corresponding Ford circles
are tangent (while all other circles in the family are disjoint).

\begin{figure}[!h]
\begin{center}
 \includegraphics[width=0.6\linewidth]{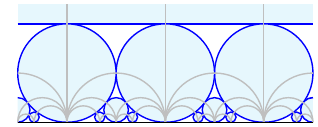}       
\caption{Ford circles}
\label{ford}
\end{center}
\end{figure}

\begin{remark}
  Consider Ford circles as horocycles at every vertex of the Farey graph. Then, given two points  $u,v\in \QQ$,  where  $u=p/q, v=r/s$ provides the relation between the determinant $\begin{pmatrix}p&r\\q&s \end{pmatrix}$
  and Penner's $\lambda$-length $\lambda_{u,v}$, see Remark~\ref{rem:lambda=det}.

\end{remark}

\section{3D Farey graph: tetrahedral graph $\mathcal T$}
\label{3D Farey graph: tetrahedral graph}

Recall that we denote $\sigma=e^{i\pi/3}=\frac{1+i\sqrt 3}{2}$.
Let  $\Q(\sigma)$ be the  field of fractions of the ring  $\Z[\sigma]$.

\subsection{Tetrahedral graph  $\mathcal T$ and its symmetries}
Consider a regular ideal tetrahedron $T\subset \H^3$ with vertices  $(0,1, \sigma, \infty)$.
Let $H$ be the group generated by reflections with respect to the faces of $T$.

\vspace{2mm}

The group $H$ acts discretely on $\H^3$ and $T$ is a fundamental domain for the action. Throughout the paper we call all images of $T$ under the action of $H$ {\em fundamental tetrahedra}.

\begin{definition}[Tetrahedral graph, tetrahedral tiling]
\label{tetrahedral_graph}
Denote by $\mathcal T$ the tiling of $\H^3$ by tetrahedra $hT$, $h\in H$, and denote by $\mathcal T_0$,
$\mathcal T_1$, $\mathcal T_2$, and  $\mathcal T_3$,
 the sets of all vertices, edges, faces, and tetrahedra themselves of the tetrahedra in the tiling respectively.
\end{definition}

The following proposition is due to Bianchi~\cite{B}.

\begin{prop}
\label{prop_Bi}  
The set of vertices of $\mathcal T$ is $\QQ(\sigma)$. The symmetry group  of $\mathcal T$ is the Bianchi group $Bi(3)=\PGL_2(\Z[\sigma])\rtimes \langle \tau \rangle$, where $\tau$ is the complex conjugation. The group  $\PGL_2(\Z[\sigma])$ of orientation-preserving symmetries of $\mathcal T$  acts transitively  on $\mathcal T_0$ and  $\mathcal T_1$.

\end{prop}

\subsection{Edges in $\mathcal T$ and det-length}
\label{sec-det}

First, we recall the notion of irreducible fractions.

\begin{definition}[Irreducible fractions in $\QQ(\sigma)$]
  A {\em fraction} is an expression $\frac{p}{q}$, $p,q\in \Z[\sigma]$.
  A fraction $\frac{p}{q}$ is {\it irreducible} if for any $k\in \Z[\sigma]$ such that $p=kp',q=kq'$ with $p',q'\in \Z[\sigma]$ one has $|k|=1$.

\end{definition}

Note that the ring $\Z[\sigma]$ has six units, i.e. the six powers of $\sigma$.
Thus, if the fraction $p/q\in \QQ(\sigma)$ is irreducible, then
the other five irreducible fractions  $p\sigma^{i}/q\sigma^{i}$  for $i=1,\ldots, 5$ represent the same element of $\QQ(\sigma)$.

\vspace{2mm}

\begin{remark}  
  \label{irreducible}
It is an easy observation that $\PGL_2(\Z\[\sigma\])$ takes irreducible fractions to irreducible ones.
\end{remark}

\begin{definition}[Det-length]
 \label{det-length}  
Given  two irreducible fractions $z_i=\frac{p_i}{q_i}$, $p_i,q_i\in \Z[\sigma]$, $i=1,2$, define {\it det-length} $l(z_1,z_2)$ by
$$l(z_1,z_2)=|\det \begin{pmatrix} p_1& p_2\\q_1&q_2 \end{pmatrix}|.
$$

\end{definition}

\begin{remark}
 Once restricted to the real line, the $\det$-length coinsides with the integer trigonometric function of sine (see, e.g.,~\cite{K}).
\end{remark}
 
Although the following proposition is well known, we give a proof for completeness.

\begin{prop}
  \label{det-l}
  The group $\PGL_2(\Z[\sigma])$ preserves the det-length $l$.

\end{prop}

\begin{proof}
Let  an element of $\PGL_2(\Z[\sigma])$ be represented by $M= \begin{pmatrix} a&b\\c&d \end{pmatrix}\in GL_2(\Z[\sigma])$.  Let $p/q$ and $m/n$ be two irreducible fractions, $p,q,m,n\in \Z[\sigma]$. Then
 \begin{multline*}
 l(M \begin{pmatrix} p\\q \end{pmatrix},M \begin{pmatrix} m\\n \end{pmatrix})=
  |\det  \begin{pmatrix} ap+bq&am+bn\\cp+dq&cm+dn \end{pmatrix}|=\\
  |\det \begin{pmatrix} a&b\\c&d \end{pmatrix}||\det \begin{pmatrix} p&m\\q&n \end{pmatrix}|=1\cdot l(\begin{pmatrix} p\\q \end{pmatrix}, \begin{pmatrix} m\\n \end{pmatrix}),
  \end{multline*}
which shows that the det-length is preserved.

\end{proof}

 Now, we give a  complete description of the edges in $\mathcal T$.

\begin{prop}
\label{edges}
Let $\frac{p_i}{q_i}$, $p_i,q_i\in\Z[\sigma] $, $i=1,2$, be two irreducible fractions. Then the line connecting $\frac{p_1}{q_1}$ to  $\frac{p_2}{q_2}$ is an edge of some tetrahedron $hT$, $h\in H$, if and only if
$l(\frac{p_1}{q_1},\frac{p_2}{q_2})=1.$

\end{prop}

\begin{proof}
  First, suppose that $p_1/q_1=\infty$. As $p_1/q_1$ is irreducible, we have $q_1=0$ and $p_1= \sigma^k $ for $k\in\{0,1,\dots,5\}$. So,
$$ |\det  \begin{pmatrix} p_1& p_2\\q_1&q_2 \end{pmatrix}|=|\det  \begin{pmatrix} 1& p_2\\0&q_2 \end{pmatrix}|=|q_2|,
$$
which implies that the determinant is a unit if and only if $p_2/q_2\in \Z[\sigma]$.
On the other hand,
one can see directly from the tiling of $\H^3$ by ideal tetrahedra that  the points in $\QQ(\sigma)$ connected to $\infty$ by an  edge of some tetrahedron $hT$, $h\in H$, are precisely ones lying in $\Z[\sigma]$.

Next, consider arbitrary irreducible fractions $p_1/q_1$ and $p_2/q_2$. According to Proposition~\ref{prop_Bi}, there exists  an element $f\in \PGL_2(\Z[\sigma])$ taking $p_1/q_1\in\QQ(\sigma)$ to $\infty$. 
In view of Remark~\ref{irreducible}, the consideration above implies that the statement holds for $f(p_1/q_1)$ and $f(p_2/q_2)$. Since $f$ preserves the det-length (Proposition~\ref{det-l}), we conclude that the statement holds for $p_1/q_1$ and  $p_2/q_2$ as well.

\end{proof}

\begin{remark}
  \label{arrangement}
  Consider the hyperbolic plane $\Pi$ in $\H^3$ containing the line ${\mathrm {Im}}  (z)=0$, and let $\mathcal F$ be the classical Farey graph lying in $\Pi$. Propositions~\ref{det-l} and~\ref{edges} imply that $\mathcal T$ can also be constructed as the orbit of $\mathcal F$ under the action of the Bianchi group $Bi(3)$ (this follows closely the construction of {\em Schmidt arrangement} for any imaginary quadratic field, see~\cite{St}). Based on that, the description of faces and fundamental tetrahedra of $\mathcal T$ (up to the action of the Bianchi group) can be deduced from~\cite{Sch}. In the rest of this section we give an explicit description of   faces and fundamental tetrahedra of $\mathcal T$  in terms of {\em Farey addition}.    

  \end{remark}

\subsection{Faces in $\mathcal T$ and  symmetric Farey addition}
\label{sec-faces}
\subsubsection{Symmetric Farey addition in $\mathcal T$}

As we have already mentioned in Section~\ref{sec_Farey_add}, 
for the adjacent vertices $\frac{p}{q},\frac{r}{s}\in \QQ$ of the classical Farey graph
the third  point adjacent to both (and lying between them) is given by the Farey addition $\frac{p}{q}\oplus \frac{r}{s}= \frac{p+ r}{q+ s}$. The other vertex of the Farey graph adjacent to both  $\frac{p}{q}$ and $\frac{r}{s}$ is given by  $\frac{p}{q}\ominus \frac{r}{s}=\frac{p- r}{q- s}$. 

\vspace{2mm}

Define  the {\it symmetric Farey sum} of  $\frac{p}{q},\frac{r}{s}\in \QQ$ as
 $$
 \frac{p}{q}\oslash \frac{r}{s}= \Big\{ \frac{p+ r}{q+ s},  \frac{p- r}{q- s}  \Big\}.
 $$
 The symmetric Farey sum provides the set of all vertices that form all fundamental triangles with the given edge joining  $\frac{p}{q}$ and $\frac{r}{s}$.
 
A symmetric Farey sum has a natural generalisation to $\QQ(\sigma)$.

\begin{definition}
Let $p/q$ and $r/s$ be irreducible fractions in $\QQ(\sigma)$.
The {\it symmetric Farey sum} of $p/q$ and $r/s$ is the following set
$$
\frac{p}{q}\oslash \frac{r}{s}= \Big\{ \frac{p+\sigma^i r}{q+\sigma^i s} \in  \QQ(\sigma)\mid  i=0,1,2,3,4,5\Big\}.
$$

\end{definition}

\begin{remark}
Note that the resulting set does not depend on the choice of irreducible fractions representing  $p/q$ and $r/s$.
The definition may look asymmetric (the second summand multiplied by $\sigma^i$ while the first is not), however, it is symmetric in   $\QQ(\sigma)$.

\end{remark}

\subsubsection{Faces in  $\mathcal T$}

By the construction, all the faces in $\mathcal T$ are triangles. Below, we describe them  in terms of 
symmetric Farey summation.

\begin{prop}
\label{farey triangles}
Three points  $\alpha,\beta, \gamma \in  \Q(\sigma)$ are the vertices of a triangle in the graph $\mathcal T$
 if and only if $l(\alpha,\beta)=1$ and 
$$
\gamma \in \alpha\oslash \beta.
$$
\end{prop}

\begin{remark}
There are precisely 6 triangles adjacent to any edge of $\mathcal T$.
\end{remark}

\begin{proof}
Without loss of generality (keeping in mind Proposition~\ref{prop_Bi}) we set $\alpha=\frac{1}{0}$ and $\beta=\frac{0}{1}$.
Then let $\gamma=\frac{p}{q}$.

By Proposition~\ref{edges} we have $l(\alpha,\gamma)=1=l(\beta,\gamma)$, and therefore $q=\sigma^i$ and  $p=\sigma^j$ for some $i,j$.
It is clear that the sets $\{\frac{\sigma^j}{\sigma^i}\mid  i,j=0,\dots,5\}$ and $ \alpha\oslash \beta$ coincide.
\end{proof}

\subsection{Fundamental tetrahedra in $\mathcal T$}
We now describe quadruples of vertices of 
 the fundamental tetrahedra in $\mathcal T$.

 \begin{prop}
\label{prop-tetr}   
Four points  $\alpha,\beta, \gamma, \delta\in  \QQ(\sigma)$ are the vertices of one fundamental tetrahedron in $\mathcal T$
if and only if they are of the form
$$
\frac{p}{q}, \quad \frac{r}{s}, \quad \frac{p+\sigma^ir}{q+\sigma^i s}, \quad \frac{p+\sigma^{i+1}r}{q+\sigma^{i+1} s}
$$
for some $p,q,r,s$ satisfying $l(\frac{p}{q}, \frac{r}{s})=1$, where $i\in\Z$ is considered modulo $6$.
\end{prop}

\begin{proof}
First, we check  that the two fundamental tetrahedra containing the triangle  $\alpha=\frac{1}{0}$, $\beta=\frac{0}{1}$, $\gamma=\frac{1}{1}$ are exactly of the required shape.
Indeed,  the condition of  Proposition~\ref{edges} for the edges $(\alpha,\delta)$ and  $(\beta,\delta)$ imply
$\delta=\frac{\sigma^i}{\sigma^j}$. The condition for   $(\gamma,\delta)$ implies that $|\sigma^i-\sigma^j|=1$, which is satisfied if and only if $i=j\pm 1$  (modulo $6$). So, either $\delta=\frac{\sigma}{1}$ or  $\delta=\frac{1}{\sigma}$.
This means that $(\alpha,\beta,\gamma,\delta)$ coincides with either  $(\frac{1}{0},\frac{0}{1},\frac{1+0\cdot \sigma^0}{0+1\cdot \sigma^0},\frac{1+0\cdot \sigma^{-1}}{0+1\cdot \sigma^{-1}})$ or $(\frac{1}{0},\frac{0}{1},\frac{1+0\cdot \sigma^0}{0+1\cdot \sigma^0},\frac{1+0\cdot \sigma}{0+1\cdot \sigma})$.

By Proposition~\ref{farey triangles} any triangle in $\mathcal T$ can be written as
$(\frac{p}{q},  \frac{r}{s},  \frac{p+\sigma^ir}{q+\sigma^i s})$.
This triangle can be obtained from $(\frac{1}{0}, \frac{0}{1},\frac{1}{1})$ by applying a transformation
$\begin{pmatrix}
  p&\sigma^ir\\
  q&\sigma^is\\
\end{pmatrix}
\in \PGL_2(\Z[\sigma])$.
This map takes $\frac{\sigma}{1}$ and  $\frac{1}{\sigma}$ to  $\frac{p+\sigma^{i+1}r}{q+\sigma^{i+1} s}$ and  $\frac{p+\sigma^{i-1}r}{q+\sigma^{i-1} s}$ respectively.
So, we get the required statement (up to swapping $\gamma$ and $\delta$).
\end{proof}

Proposition~\ref{prop-tetr}  can be reformulated in the following way.

\begin{cor}\label{tetraT}
Four points  $\alpha,\beta, \gamma, \delta\in  \QQ(\sigma)$ are the vertices of one fundamental tetrahedron in $\mathcal T$
if and only if they are of the form
$$
\frac{p}{q},
\quad
\frac{r}{s},
\quad
\frac{p+r}{q+s},
\quad \hbox{and} \quad 
\frac{p+\sigma r}{q+ \sigma s},
$$
where $|ps-qr|=1$.
\end{cor}

\section{Relations on $\lambda$-lengths}
\label{Relations on lambda-lengths}

In this section we derive some relations on $\lambda$-lengths which can be considered as $3$-dimensional analogues of the Ptolemy relation. We start by recalling Penner's definition of  $\lambda$-lengths and one of the proofs of the hyperbolic version of the Ptolemy relation. Then we continue by using similar ideas in $3$-dimensional setting.

\subsection{$\lambda$-lengths and Ptolemy relation}
\label{l-length-intro}

The following definition was introduced by Penner~\cite{P}.

\begin{definition}[$\lambda$-length]
  Given two points $A,B\in \partial \H^n$ together with the choice of horoballs  $h_A$ and $h_B$ centred at $A$ and $B$ respectively, the     {\em $\lambda$-length} $\lambda_{AB}$ of the segment $AB$  is defined as $\exp(d/2)$, where $d$ is the (signed) distance between the horospheres $h_A$ and $h_B$.
Here $d=0$ when the horoballs  $h_A$ and $h_B$ are tangent and $d<0$  when the horoballs intersect non-trivially (in the latter case $d$ is defined as negative distance $d(a,b)$ between the intersection points $a=AB\cap h_A$ and $b=AB\cap h_B$  of the line $AB$ with the horoballs).

\end{definition}

The computation in the next example is a particular case of~\cite[Corollary~4.2]{P2}. 

\begin{example}[Computing $\lambda$-lengths]
  \label{z^2/2h}
  Consider the upper halfplane model of the hyperbolic plane $\H^2$ with complex coordinate $z$.

\begin{itemize}
\item[(a)]  Let $A=\infty, B=0$, let $h_A$ be given by equation  ${\mathrm {Im}}  (z)=1$ (horizontal line) and $h_B$ be given by (Euclidean) circle of radius $r$  tangent to the absolute  ${\mathrm {Im}}  (z)=0$, see Figure~\ref{inv}, left.
  Then $$d=d(i,2ri)=
  \ln\frac{1}{2r}, $$ and thus
  $$\lambda_{AB} = \exp({d/2})=
  \frac{1}{\sqrt{2r}}.$$

\item[(b)] Let $A=0$, $B=z$, let $h_A$ be given by Euclidean circle of radius $1/2$ and $h_B$ be given by Euclidean circle of radius $r$, then $\lambda_{AB}=\frac{|z|}{\sqrt{2r}}$.
This is easy to check using inversion with respect
 to the unit circle centred at the origin and applying the result of (a), see Figure~\ref{inv} middle and right.

\item[(c)] Result of (b) implies Ptolemy relation for ideal quadrilaterals in $\H^2$ (see Example~\ref{ex_Pt} below).
\end{itemize}
\end{example}

\begin{figure}[!h]
\begin{center}
 \epsfig{file=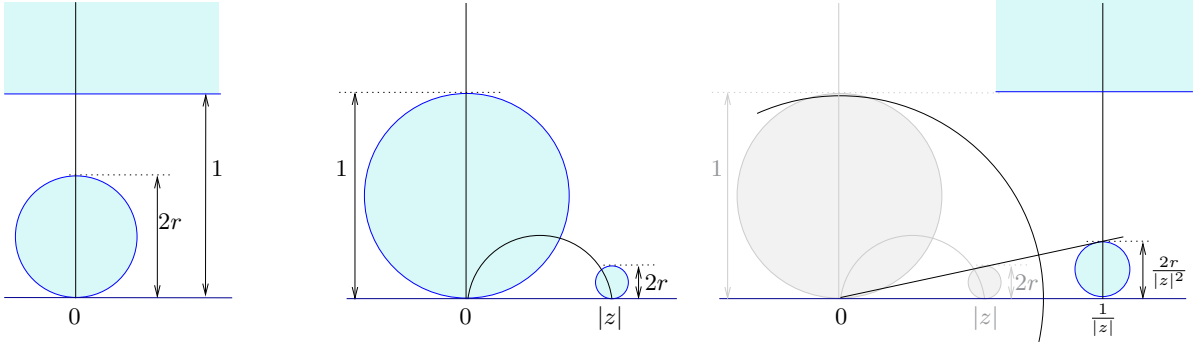,width=0.99\linewidth}
 \put(-372,63){\scriptsize $1$}
 \put(-392,43){\scriptsize $2r$}
 \put(-426,7){\scriptsize $0$}
 \put(-325,63){\scriptsize $1$}
 \put(-208,20){\scriptsize $2r$}
 \put(-278,7){\scriptsize $0$}
 \put(-226,7){\scriptsize $|z|$}
 \put(-183,63){\scriptsize {\color{Gray} $1$}}
 \put(-68,20){\scriptsize  {\color{Gray} $2r$}}
 \put(-136,7){\scriptsize $0$}
 \put(-84,7){\scriptsize \color{Gray} $|z|$}
 \put(-40,7){\scriptsize  $\frac{1}{|z|}$}
 \put(-17,25){\scriptsize  $\frac{2r}{|z|^2}$}
 \caption{To Example~\ref{z^2/2h}: (a) on the left, (b) in the middle, and inversion with respect to the unit circle on the right.}
\label{inv}
\end{center}
\end{figure}

\begin{example}[Ptolemy relation]
 \label{ex_Pt}
 It was shown by Penner in~\cite{P} that $\lambda$-lengths satisfy the  Ptolemy relation:
 
 {\it Given an ideal quadrilateral $ABCD$ $($i.e. $A,B,C,D\subset \partial \H^2$$)$ with any choice of horocycles at points $A,B,C,D$, the $\lambda$-lengths of sides and diagonals of $ABCD$  satisfy}
  \begin{equation}
    \label{Pt}
   \lambda_{AC} \lambda_{BD} =\lambda_{AB} \lambda_{CD}+ \lambda_{BC} \lambda_{AD}.
  \end{equation}

We prove this  relation in three steps:

\begin{itemize}
\item[a:] {\it Reducing to the quadrilateral  $(0,1,\infty, x)\subset \H^2$.} By using the isometry group  of $\H^2$
  we can assume that a quadrilateral $ABCD$ has vertices $0,1,\infty, x$, we will denote it by $(0,1,\infty, x)\subset \H^2$.

\item[b:] {\it Choosing the horocycles.} Notice that a choice of horocycles does not affect validity of~(\ref{Pt}).
Indeed,
 changing a horocycle at one vertex $Z\in\{A,B,C,D\}$  changes the length $d$ of every edge incident to $Z$ by the same number $\gamma$, and hence the corresponding $\lambda$-lengths are multiplied by the same number $e^{\gamma/2}$, which preserves~(\ref{Pt}) as the relation is homogeneous. In particular, we may assume that three of the horocycles are mutually tangent.

\item[c:] {\it Proof for  quadrilateral $(0,1,\infty, x)\subset \H^2$, with three mutually tangent horocycles.}
It is now sufficient to show~(\ref{Pt}) for the
 ideal quadrilateral $(0,1,\infty, x)\subset \H^2$, with $x\in \R$, $0<x<1$, and the horocycles chosen as follows (see also Figure~\ref{pt}, right). Let ${\mathrm {Im}} (z)=1$ be the horocycle at $\infty$, $|z-i/2|=1/2$ and  $|z-1-i/2|=1/2$  be the horocycles at $0$ and $1$ respectively, and  let $|z-x-ir|=r$ be the horocycle at $x$. Then the three of the horocycles are mutually tangent, so that $\lambda_{0,\infty}=\lambda_{1,\infty}=\lambda_{01}=1$. We can also compute using the result of Example~\ref{z^2/2h} (b):
  $$ \lambda_{0,x}=\frac{x}{\sqrt{2r}}, \qquad  \lambda_{1,x}=\frac{1-x}{\sqrt{2r}}, \qquad  \lambda_{\infty,x}=\frac{1}{\sqrt{2r}},
$$
so that we get $\lambda_{0,x}+\lambda_{x,1}=\lambda_{x,\infty}$, which is exactly the Ptolemy relation for the quadrilateral
$(0,1,\infty, x)$.

\end{itemize}

\end{example}

\begin{figure}[!h]
\begin{center}
  \epsfig{file=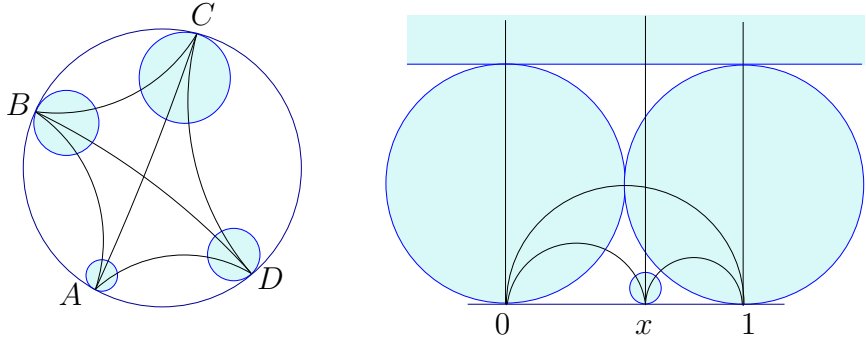,width=0.7\linewidth}
 \put(-140,-10){$0$}
 \put(-87,-10){$x$}
 \put(-47,-10){$1$}
 \put(-230,7){$D$}
 \put(-305,1){$A$}
 \put(-325,73){$B$}
 \put(-255,107){$C$}
\caption{Ptolemy Theorem and one of its proof}
\label{pt}
\end{center}
\end{figure}

\begin{remark}[$\lambda$-lengths and det-lengths]
\label{rem:lambda=det}  
One can observe that det-lengths $l(u,v)$ between points $u,v\in \QQ$ (see Definition~\ref{det-length})
satisfy Ptolemy relation (e.g. by combining results of~\cite{CC} and~\cite{MGOT}). 
Computing $\lambda$-length between points of $\QQ$ with respect to Ford circles shows that both $\lambda$-lengths and det-lengths of the sides of fundamental triangles of the Farey tessellation are equal to $1$. Applying Ptolemy relation iteratively we conclude that
$$ \lambda_{u,v}=l(u,v)
$$
for every $u,v\in \QQ$.
\end{remark}  

\subsection{Ptolemy relation in $\mathcal T$}
\label{sigma-pt}

\begin{definition}[Standard horospheres at $\mathcal T_0$: counterpart of Ford circles]
  Consider the fundamental tetrahedron $T=(0,1,\sigma,\infty)$ of $\mathcal T$, choose four pairwise tangent horospheres at its vertices (i.e. horospheres represented by three balls of radia $1/2$ and a horizontal plane at the height 1). Notice that any isometry of $T$ takes the four horospheres to the same four horospheres.
Now, we can use the action of $\PGL_2(\Z[\sigma])$ to choose a horosphere at each point of   $\mathcal T_0=\QQ(\sigma)$ (so that at every fundamental tetrahedron the four horospheres are mutually tangent). This choice of horospheres at $\QQ(\sigma)$ will be called {\it standard}.

\end{definition}

\begin{remark}
Definition of standard horospheres is equivalent to the definition of {\em Ford spheres} given by Northshield~\cite{N}.
  
  \end{remark}

In this section, we  assume by default the standard choice of horospheres at $\mathcal T_0$.
 This assumption simplifies the reasonings and does not affect the statements (similarly to step (b) in Example~\ref{ex_Pt}), as the  relations considered below are all homogeneous.
 The only exception is  Theorem~\ref{lambda=l}, where the standard choice of horospheres is essential.

\begin{theorem}
\label{1}
Let $A_1A_2A_3A_4$ be a fundamental tetrahedron with vertices in  $\QQ(\sigma)$,  choose any $X\in  \QQ(\sigma)$ distinct from $A_i$.
Let $\lambda_i=\lambda_{XA_i}$ be the $\lambda$-length of $XA_i$, $i=1,\dots,4$.
Then
\begin{equation}
\label{lambda4}
\sum\limits_{i=1}^4 \lambda_i^4=\sum\limits_{1\le i<j\le 4} \lambda_i^2\lambda_j^2.
\end{equation}

\end{theorem}

\begin{proof}
 The settings of the theorem are illustrated in Figure~\ref{ind}, left. Applying an isometry,
we may assume that $A_1A_2A_3A_4=(0,1,\sigma,\infty)$ and that $X=z\in \QQ(\sigma)$.
Let the horosphere at $z$ be represented by Euclidean sphere of radius $r/2$. Then  from Example~\ref{z^2/2h}(b) we get that
$$ \lambda_i= \frac{|z-A_i|}{\sqrt{2r}}, \  i=1,2,3, \qquad  \qquad \lambda_4=\frac{1}{\sqrt{2r}}.
$$
It is now sufficient to check that the values $a_i=|z-A_i|^2, i=1,2,3$ and $a_4=1$ satisfy the equation
$$ \sum\limits_{i=1}^4 a_i^2=\sum\limits_{1\le i<j\le 4} a_ia_j.
$$
This  is a straightforward computation by using the following substitutions: $a_1=z\bar z$, $a_2=(z-1)(\bar z-1)$, and $a_3=(z-\sigma)(\bar z-\bar \sigma)$.

\end{proof}

\subsection{\bf $\lambda$-lengths and Soddy-Gosset theorem}
As was noted to the authors by  Arthur Baragar and Ian Whitehead, 
Equation~\ref{lambda4}
can be considered as a corollary of
the Soddy-Gosset theorem  relating the radia of five mutually tangent  spheres in the Euclidean space  (see e.g.~\cite{LMW}). Given $n+2$ mutually tangent spheres in $\R^n$ of radia $r_i$, $i=1,\dots, n+2$, denote $k_i=1/r_i$. Then the Soddy-Gosset theorem  asserts that
  \begin{equation}
  \label{SG}
  (\sum\limits_{i=1}^{n+2} k_i)^2=n\sum\limits_{i=1}^{n+2} k_i^2.
\end{equation}

The plane version of this theorem (4 mutually tangent circles in $\E^2$) is called the  Descartes Circle theorem. We will use the version with $n=3$ and five spheres, one of which is of infinite radius, and hence has $k_5=0$. Equation~\ref{SG} is then reduced to
 \begin{equation}
  \label{k}
  \sum\limits_{1\le i< j\le 4} k_ik_j=\sum\limits_{i=1}^{4} k_i^2,
\end{equation}
cf. Equation~\ref{lambda4}.

Now, consider the configuration of points $A_1,A_2,A_3,A_4$ and $X$ as in Theorem~\ref{1}: here, $A_1A_2A_3A_4$ is a fundamental tetrahedron of $\mathcal T$ (or any other regular ideal tetrahedron, i.e. an image of the fundamental tetrahedron under any isometry of $\H^3$) and $X$ is any point on the absolute of $\H^3$. After applying an isometry we may assume that $X=\infty$ (we use the upper halfspace model with coordinates $\{ (z,t) \ | \ z\in \C, t\in \R_+ \}$).

Next, we choose four mutually tangent horospheres  centred at points $A_1,A_2,A_3,A_4$ (this is possible as $A_1A_2A_3A_4$ is a regular ideal tetrahedron). The four horospheres together with the plane representing the absolute play the role of the five mutually tangent spheres in the Soddy-Gosset theorem. Let $r_i$ be the radius of the horosphere centred at $A_i$. Applying the Soddy-Gosset theorem as in Equation~\ref{k}, we get
 \begin{equation}
  \label{h}
  \sum\limits_{1\le i< j\le 4} \frac{1}{r_i}\frac{1}{r_j}=\sum\limits_{i=1}^{4} \frac{1}{r_i^2}.
\end{equation}

We also choose the horosphere $h_X$ centred at $X$ as the plane given by $t=1$  and denote by $\lambda_i$ the $\lambda$-lengths of $A_iX$. In view of the computation in Example~\ref{z^2/2h}(a) we have
$$\lambda_i=\frac{1}{\sqrt{2r_i}},$$
so, replacing $1/r_i$ by $2\lambda_i^2$ in Equation~\ref{h} we get Equation~\ref{lambda4}.

\subsection{Det-lengths and $\lambda$-lengths}

In this section, we show that det-lengths between points of $\QQ(\sigma)$  coincide with $\lambda$-lengths with respect to standard horospheres.

We start by showing that det-lengths satisfy the counterpart of Relation~\ref{lambda4}.

\begin{theorem}
\label{2}
For points  $A_1A_2A_3A_4$ and $X$ as in Theorem~\ref{1}, define $l_i:=l_{XA_i}$. Then
$$
 \sum\limits_{i=1}^4 l_i^4=\sum\limits_{1\le i<j\le 4} l_i^2l_j^2.
$$

\end{theorem}

\begin{proof}
  As in the proof of Theorem~\ref{1},  we  assume that $A_1A_2A_3A_4=(0,1,\sigma,\infty)$ and that $X=z\in \Q(\sigma)$.
  We can write  $z=\frac{p}{q}=\frac{p_1+p_2\sigma}{q_1+q_2\sigma}$, $p_1,p_2,q_1,q_2\in\Z$, and hence obtain
$$  l_1=l_{0,z}=|p|,  \qquad   l_2=l_{1,z}=|p-q|, \qquad   l_3=l_{\sigma,z}=|p-q\sigma|, \qquad   l_4=l_{\infty,z}=|q|.
$$
Taking into account that $|p|^2=p\bar p= p_1^2+p_2^2 +p_1p_2$  (and that there are similar expressions for $|q|^2,|p-q|^2$ and $|p-\sigma q|^2$), one can easily check the identity claimed in the theorem.

\end{proof}

\begin{remark}
\label{det-pt}
Theorem~\ref{2} is a  counterpart of the 2-dimensional statement mentioned in Remark~\ref{rem:lambda=det}.

\end{remark}

\begin{remark} Since the equations above are homogeneous,
  Theorem~\ref{2} is not affected when the fraction for $z$ is not irreducible.
Indeed, when we change the horosphere (resp. multiplying the numerator and denominator by the same factor), each of the summands in the sum is multiplied by the same factor.

\end{remark}

\begin{remark}
\label{two roots}
The relation in Theorems~\ref{1}  is quadratic with respect to $\lambda_4^2$. It has two positive roots, these roots have the following geometrical meaning. Let $A_4'$ be the point obtained from $A_4$ by the reflection with respect to the plane $A_1A_2A_3$ (i.e. $A_1A_2A_3A_4'$ is the fundamental tetrahedron adjacent to $A_1A_2A_3A_4$ along  $A_1A_2A_3$). We assume that $A_4'$ lies in the same half-space with respect to $A_1A_2A_3$ as $X$ while $A_4$ lies in the other half-space, see Figure~\ref{ind}, middle.

Denote $\lambda_4'=\lambda_{XA_4'}$ the corresponding $\lambda$-length. Then $\lambda_4$ and $\lambda_4'$ are the two roots of the quadratic  equation in Theorem~\ref{1} (as in Theorem~\ref{1}, we do not make any assumptions on the position of point $X$).

\end{remark}

\begin{figure}[!h]
\begin{center}
  \epsfig{file=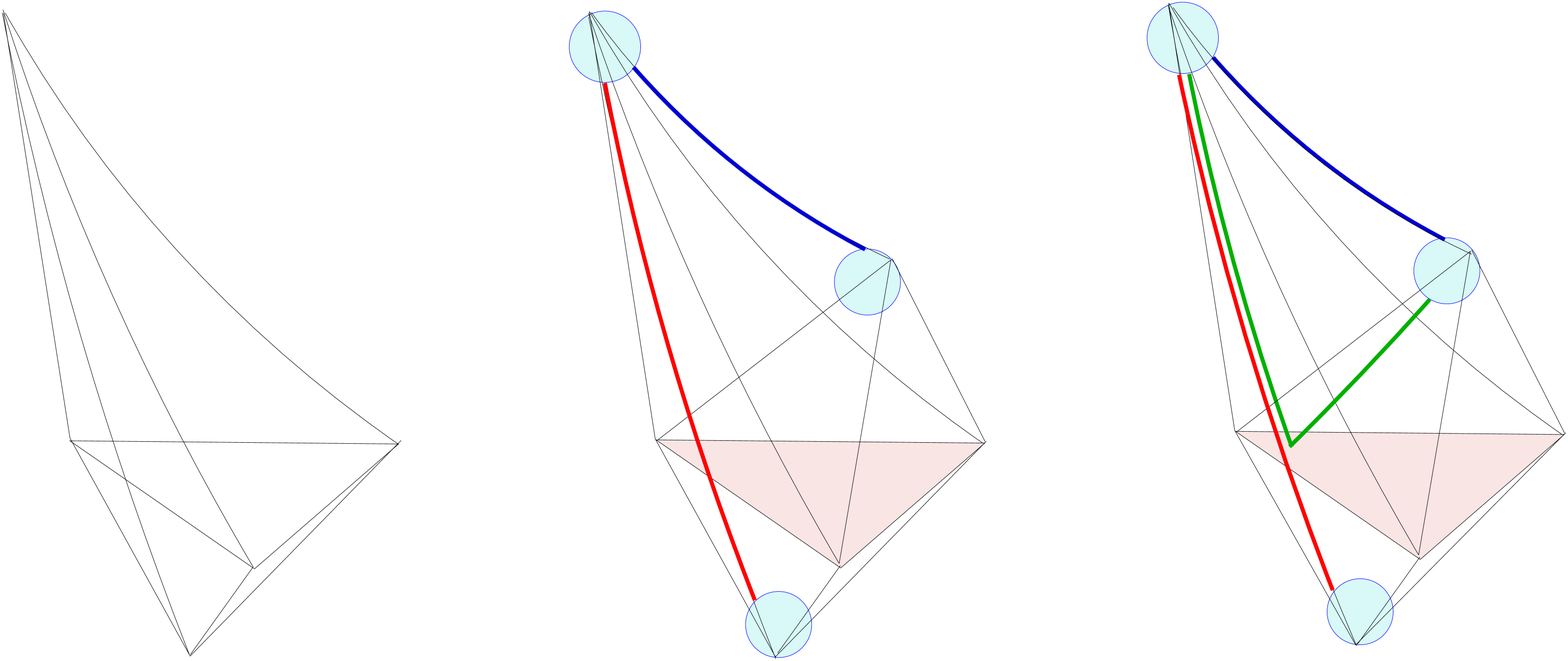,width=0.99\linewidth}
 \put(-460,183){$X$}
 \put(-295,183){$X$}
 \put(-445,60){$A_1$}
 \put(-335,60){$A_2$}
 \put(-375,20){$A_3$}
 \put(-395,-5){$A_4$}
 \put(-225,-6){$A_4=Y$}
 \put(-195,117){$A_4'=Y'$}
 \put(-77,55){$P$}
  \put(-136,183){$X_0$}
\put(-25,117){$X_{i}$}
 \put(-120,60){$X_{i+1}$}
 \put(-0,60){$X_{i+2}$}
 \put(-35,25){$X_{i+3}$}
 \put(-59,-6){$X_{i+4}$}
\put(-452,120){\color{Plum} $\lambda_1$}
 \put(-400,120){\color{Plum}$\lambda_2$}
 \put(-417,100){\color{Plum}$\lambda_3$}
 \put(-405,23){\color{Plum}$\lambda_4$}
 \put(-235,23){\color{red}$\lambda_4$}
 \put(-239,143){\color{blue}$\lambda_4'$}
\caption{Notation for Theorems~\ref{1},~\ref{lambda=l}, and~\ref{lin}}
\label{ind}
\end{center}
\end{figure}

Remark~\ref{two roots} gives rise to the following theorem.

\begin{theorem}
  \label{lambda=l}
  Let $X,Y\in \QQ(\sigma)$ be two  irreducible fractions. Let the standard horosphere be chosen at every point of $\QQ(\sigma)$.
Then $\lambda_{XY}=l_{XY}$.

\end{theorem}

\begin{proof}
  We consider two cases: either the geodesic $XY$  lies in a plane containing a face of a fundamental tetrahedron of $\mathcal T$ or it does not.

  First, suppose that $XY$ lies in a plane $\Pi$ containing a face of  a fundamental tetrahedron of $\mathcal T$. The tessellation $\mathcal T$ induces Farey triangulation on $\Pi$. Denote by $S$ the set  of triangles whose interior is intersected by the geodesic $XY$, let $N=|S|$. We  proceed by induction on $N$.
  If $N=0$, then  $XY$ is an edge of some triangle (and edge of $\mathcal T$)  and  $l_{XY}=1=\lambda_{XY}$  (by Proposition~\ref{edges}  and since the horospheres are mutually tangent). Otherwise, $N\ge 2$.

Let $T_Y\in S$ be the triangle with vertex $Y$. Denote by $Y_1$ and $Y_2$ the other two vertices of $T_Y$. Then the set of triangles intersected by geodesics $XY_1$ and $XY_2$ are subsets of $S$ not containing $T_Y$, so their det-lengths coincide with $\lambda$-lengths by the induction assumption.

Consider the quadrilateral with vertices $X, Y_1,Y_2,Y$.  As we have seen above, for all its edges  det-lengths coincide with $\lambda$-lengths, as well as for $Y_1Y_2$. Then both $\lambda_{XY}$ and $l_{XY}$ satisfy the same Ptolemy relation (see Remark~\ref{det-pt}), and thus coincide as well.

  Now, suppose that $XY$ does not lie in the plane of a face of any fundamental tetrahedron of $\mathcal T$, the proof is very similar to the two-dimensional case. Denote by $S$ the set  of (closed) fundamental tetrahedra intersected by $XY$ in an interior point of $\H^3$, let $N=|S|$. Again, we  proceed by induction on $N$.

  Let $T_Y\in S$ be the fundamental tetrahedron with vertex $Y$. Denote by $Y_1,Y_2,Y_3$ the other three vertices of $T_Y$. For every $i=1,2,3$ either the set of tetrahedra intersected by the geodesic $XY_i$ is a subset of $S$ not containing $T_Y$, or $XY_i$ belongs to the plane of a face of some fundamental tetrahedron. By the induction assumption (and the two-dimensional case), det-lengths of all $XY_i$ coincide with their $\lambda$-lengths.

  Consider  the triangular bipyramid with vertices $X$, $Y$ and $Y_i$. For all its edges  det-lengths coincide with $\lambda$-lengths. By Theorems~\ref{1} and~\ref{2},
  the values $\lambda_{XY}$ and $l_{XY}$ satisfy the same quadratic equation, and we are left to show that it should be the same root.

Let $T_Y'$ be  the fundamental tetrahedron adjacent to $T_Y$ along the face $Y_1Y_2Y_3$, and let $Y'$ be its fourth vertex. From  Remark~\ref{two roots}  we know that $\lambda_{XY'}$ and $l_{XY'}$ are also the roots of the same quadratic equation as  $\lambda_{XY}$ and $l_{XY}$. By the inductive assumption,   $\lambda_{XY'}=l_{XY'}$ (as $XY'$ intersect a smaller number of tetrahedra than $N$). This implies that each of
  $\lambda_{XY}$ and $l_{XY}$ is the {\it  other} root of the same quadratic equation, which implies
  $\lambda_{XY}=l_{XY}$.

\end{proof}

  \begin{remark}
Theorem~\ref{lambda=l} implies that   the set of all $\lambda$-lengths of arcs with  ends at $\QQ(\sigma)$ coincides with the set of absolute values of all elements of  $\Z[\sigma]$.

\end{remark}

\begin{remark}
  \label{radius}
    There is another, purely computational, approach to Theorem~\ref{lambda=l}. Given an irreducible fraction $p/q\in\QQ(\sigma)$, the radius of the standard horosphere at $p/q$ is equal to $1/2|q|^2$ (see e.g.~\cite{R} by Rieger). Using Example~\ref{z^2/2h}(a), one can see that the $\lambda$-length between $p/q$ and $\infty=1/0$ is equal to the corresponding det-length. Now, every point of  $\QQ(\sigma)$ can be taken to $\sigma^k/0$ by an element of $PSL_2(\Z[\sigma])$, and as both   $\lambda$-length and det-length are invariant, the result follows.   

\end{remark}

\subsection{Linear relation }

Consider a geodesic $\gamma$ connecting two points $X$ and $Y$ in $\QQ(\sigma)$. Consider all fundamental tetrahedra crossed by $\gamma$. Denote $X=X_0$, let $X_1,X_2,X_3$ the other three vertices of the first tetrahedron crossed by $\gamma$, and $X_i>3$ the vertices added on the way when tetrahedra are attached one by one along $\gamma$.

\begin{theorem}
\label{lin}
  Let $b_i=\lambda^2_{X_0X_i}$. Then
  \begin{itemize}
   \item[(a)]
     $b_i+b_{i+4}=b_{i+1}+b_{i+2}+b_{i+3}$.
   \item[(b)] $b_i\in \Z$.

   \item[(c)] $b_i<b_{i+4}$.

  \end{itemize}
\end{theorem}

\begin{proof}
  In view of Theorem~\ref{1}, the values $b_i$ and $b_{i+4}$ are the two roots of the following quadratic equation
  $$ x^2-x(b_{i+1}+b_{i+2} +b_{i+3}) + b_{i+1}^2+b_{i+2} ^2+b_{i+3}^2+b_{i+2}b_{i+3}+b_{i+1}b_{i+3} +b_{i+1}b_{i+2}=0,
  $$
  which implies  $b_i+b_{i+4}=b_{i+1}+b_{i+2}+b_{i+3}$, which settles part (a).

  Applying equation  (a) iteratively we see that $b_i\in \Z$ for all $i$, as $b_1=b_2=b_3=1$ and $b_4=3$ (where the latter follows from Theorem~\ref{1}), which proves (b).

  To prove (c), notice that $X_0$ is separated from  $X_{i+4}$ by the plane $X_{i+1}X_{i+2}X_{i+3}$, and $X_i$ is the reflection image of   $X_{i+4}$ with respect to $X_{i+1}X_{i+2}X_{i+3}$. To compare $\lambda_{X_0X_{i+4}}$ with  $\lambda_{X_0X_{i}}$ notice, that the distance from the horosphere at $X_0$ to the horosphere at $X_{i+4}$ is equal to the sum of the distance from the horosphere at $X_0$ to the point $P= X_{i+1}X_{i+2}X_{i+3}\cap X_0X_{i+4}$   and the distance from $P$  to the horosphere at $X_i$, see Figure~\ref{ind}, right. The latter sum is larger than the distance from horosphere at $X_0$ to the horosphere at $X_i$, which implies   $b_{i+4}>b_i$.

\end{proof}

\begin{remark}
  Equation (a)  of Theorem~\ref{lin} is a counterpart of the following relation in the classical Farey graph:
  $$b_i+b_{i+3}=2(b_{i+1}+b_{i+2}), $$
where $b_i=\lambda^2_{X_0X_i}$ and triangles $X_iX_{i+1}X_{i+2}$ and  $X_{i+1}X_{i+2}X_{i+3}$ are fundamental ones.
\end{remark}

\subsection{3D Ptolemy relation: general formula } In this section we generalise relation (\ref{lambda4})  to any
pair of adjacent ideal tetrahedra in $\H^3$.

\begin{theorem}
\label{10}  
Let $X_1,\dots,X_5\in \CC=\partial \H^3$ be $5$ distinct points. Suppose that there are horospheres chosen in these points. Let $\lambda_{ij}=\lambda_{X_iX_j}$.  Then
$$ \sum\limits_{i<j} \lambda_{ij}^4 \lambda_{kl}^2\lambda_{lm}^2\lambda_{mk}^2= \sum\limits_{\text{{\rm cycles} $(ijklm)$}} \lambda_{ij}^2\lambda_{jk}^2\lambda_{kl}^2\lambda_{lm}^2\lambda_{mi}^2,$$
where all indices $i,j,k,l,m$ are distinct.

\end{theorem}

We use notation as in Figure~\ref{notation}. The theorem is an immediate corollary of the following lemma.

\begin{lemma}
\label{formula}
  Given two adjacent ideal tetrahedra as in Figure~\ref{notation}, one has
  \begin{multline*}\lambda^4\prod\limits_{i}t_i^2+  \sum\limits_{i} d_i^4t_i^2s_j^2s_k^2 + \sum\limits_{k} d_i^2d_j^2t_k^2s_k^4 + \lambda^2\sum\limits_i d_i^2t_i^4s_i^2\\= \sum\limits_k d_i^2d_j^2s_k^2(t_i^2s_i^2+t_j^2s_j^2) + \lambda^2\sum\limits_i d_i^2t_i^2(t_j^2s_j^2+t_k^2s_k^2).
    \end{multline*}
\end{lemma}

\begin{proof}
 Let $X_1,X_2,\dots,X_5$ be vertices of the polyhedron in   Figure~\ref{notation} (it is not important for the proof how they are matched to vertices).
  First, we apply the action of $PSL_2(\C)$ on $\H^3$ to map $X_1,X_2,X_3$ to $0,1,\infty$ (this does not change any of $\lambda_{ij}$). Suppose that $X_4$ and $X_5$ are mapped to points $z,w\in \C$. Notice that the choice of horospheres does not affect whether the equation is true or not (since the equation is homogeneous). So, we can take three tangent horospheres at $0,1,\infty$ and the two other horospheres tangent to the horosphere at $\infty$. This makes five of the $\lambda$-lengths equal to 1. The other five we compute
 using the formula from Example~\ref{z^2/2h}, and check that the obtained expressions satisfy the equation.

\end{proof}

\begin{figure}[!h]
\begin{center}
  \epsfig{file=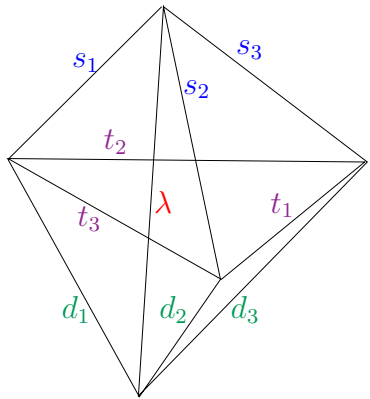,width=0.3\linewidth}
\put(-37,70){\color{Plum}$t_1$}
\put(-100,94){\color{Plum}$t_2$}
\put(-110,65){\color{Plum}$t_3$}
\put(-50,130){\color{blue}$s_3$}
\put(-70,115){\color{blue}$s_2$}
\put(-112,125){\color{blue} $s_1$}
\put(-52,30){\color{ForestGreen}$d_3$}
\put(-79,30){\color{ForestGreen}$d_2$}
\put(-116,30){\color{ForestGreen}$d_1$}
\put(-81,70){\color{red}$\lambda$}
\caption{Notation for Lemma~\ref{formula}: $\lambda$-lengths of edges}
\label{notation}
\end{center}
\end{figure}

\begin{remark}
It was noted to the authors by Ivan Izmestiev that   Theorem~\ref{10}, as well as
   similar formulae relating $\lambda$-lengths between $n+2$ points at the boundary of $n$-dimensional hyperbolic space for any $n\ge 4$, can be obtained as follows. Given an isotropic vector $u$ in the hyperboloid model of $\H^n$, every horosphere centred at $u$ can be written as $h_{u,c}=\{ x\in \H^n \mid \langle x,c u\rangle =-1/\sqrt 2\}$ for some $c\in \R_+$. Then for two points $u,v\in \partial H^n$ the $\lambda$-length  $\lambda_{u,v}$ with respect to the horospheres $h_{u,a}$ and $h_{v,b}$   
can be written as $\lambda_{u,v}=\sqrt{-\langle au,bv\rangle}$, see~\cite{P2}. Now, the formula for $n+2$ points $\{u_i\}$ can be obtained by expanding the determinant of the matrix $(\langle c_iu_i,c_ju_j\rangle)$, which vanishes since the vectors are linearly dependent. 

\end{remark}

\section{$\SL_2$-tilings over $\Z[\sigma]$}
\label{SL_2-tilings}

Taking their origin in Conway-Coxeter frieze patterns~\cite{C,CC}, $\SL_2$-tilings were introduced in~\cite{ARS} and became a topic of a rapidly growing field of studies connecting combinatorics, geometry, cluster algebras and many related domains, see the review by Morier-Genoud~\cite{MG} concerning the connections.
It was shown by Short~\cite{Sh} that all $SL_2$-tilings over $\Z$ can be classified in terms of  pairs of paths on the Farey graph $\mathcal F$.

In this section, we provide a classification of $SL_2$-tilings with entries in  $\Z[\sigma]$ in terms of pairs of paths in $\mathcal T$, generalising the results of~\cite{Sh}.

\subsection{Normalised  paths}

\begin{definition}
  
We say that $(v_i)_{i=k}^{n}$ for $k\in \{-\infty,0\}$ and $n\in \Z_+\cup \{+ \infty\}$  is a
{\it path} on $\T$ if $v_iv_{i+1}$ is an edge of $\mathcal T$ for every $i$ satisfying $k\le i <n$.

\end{definition}

We will abuse notation by  writing $(v_i)$ if the path is bi-infinite.

\begin{definition}
\label{irreducible_path}
A path $(v_i)_{i=k}^{n}$  represented by irreducible fractions $v_i=p_i/q_i$ for all $i$  is called  {\it normalised}
if 
$$
  \det
\begin{pmatrix} p_i& p_{i+1} \\ q_i&q_{i+1}
\end{pmatrix}
=1
$$
for all admissible $i$.

\end{definition}

\begin{remark}\label{18-similarity}
  Given an irreducible fraction $p_0/q_0$ representing $v_0\in \QQ(\sigma)$
  (so that we fix one of the six fractions of the form $\sigma^{k}p_0/\sigma^{k}q_0$),
  there is a unique fraction $p_1/q_1$ representing $v_1\in \QQ(\sigma)$ and satisfying  the condition for a normalised path given in Definition~\ref{irreducible_path}. Similarly  each $p_i/q_i$ in the normalised path  $(v_i)$ can be uniquely reconstructed from $p_0/q_0$. Therefore, for every nontrivial path  there are precisely six different normalisations
  (obtained from one of them by multiplying both $p_i$ and $q_i$ by $\sigma^{(-1)^i\cdot k}$).

\end{remark}

\subsection{ $\SL_2$-tilings}

\begin{definition}[$\SL_2$-tiling]
  
  Let $M=(m_{ij})$ be a bi-infinite matrix, $i,j\in \Z$, $m_{ij}\in R$, where $R$ is an  integral domain.
  $M$ is an {\it $\SL_2$-tiling} if
$$ \begin{pmatrix} m_{i,j} & m_{i,j+1} \\ m_{i+1,j} & m_{i+1,j+1}\end{pmatrix} \in \SL_2(R)$$
for any $i,j\in \Z$.

An $\SL_2$-tiling $M$ is {\it tame} if
$$\det \begin{pmatrix} m_{i,j} & m_{i,j+1}&  m_{i,j+2}\\ m_{i+1,j} & m_{i+1,j+1}& m_{i+1,j+2}\\m_{i+2,j} & m_{i+2,j+1}& m_{i+2,j+2}\end{pmatrix}=0$$
for any $i,j\in \Z$.

\end{definition}

Given two  bi-infinite normalised paths  $(u_i)$ and $(v_j)$, where $u_i=p_i/q_i$,  $v_j=r_j/s_j$, consider the numbers
\begin{equation}
\label{sl}
  m_{ij}= m(u_i,v_j)=\det \begin{pmatrix} p_i &r_j \\ q_i&s_j\end{pmatrix}.
\end{equation}

\begin{prop}[\cite{Sh}]
\label{sh} If $R=\Z$,
  the bi-infinite matrix $(m_{ij})$ is a tame $\SL_2$-tiling. Moreover,   $($\ref{sl}$)$ provides a bijection between
  tame $\SL_2$-tilings  $($modulo multiplication of all entries by $-1$$)$ and  pairs of bi-infinite normalised paths in the Farey graph $\mathcal F$ $($considered up to simultaneous action of $\SL_2(\Z)$$)$.

\end{prop}

\subsection{$\SL_2$-tilings over $\Z[\sigma]$  and pairs of paths in $\T$}
The goal of this section is to prove a counterpart of Proposition~\ref{sh} based on the paths in $\T$ (see Theorems~\ref{sigma-sh} and~\ref{sh-det}).

\begin{definition}
Consider two bi-infinite normalised paths $(p_i/q_i)$
and $(r_j/s_j)$ in $\T$.
The
{\it scalar product of two paths}
is a matrix $(m_{i,j})$ where
$$
m_{i,j}=p_ir_j+q_is_j
$$
for all  $i$ and $j$. In this case we  write  $(m_{i,j})= (p_i/q_i)\cdot(r_j/s_j).$
\end{definition}

\begin{remark}
  In terms of matrices,  the scalar product of two paths can be understood as  a matrix multiplication:
  $$
  (m_{i,j})=\begin{pmatrix} ...&...\\  p_{i-1} &q_{i-1} \\ p_i & q_i \\ p_{i+1}& q_{i+1}\\ ...&....\end{pmatrix}
  \begin{pmatrix} ...& r_{j-1}&r_j&r_{j+1}&... \\...& s_{j-1}&s_j&s_{j+1}&... \\ \end{pmatrix}.
  $$

\end{remark}

\begin{prop}
\label{paths->tiling}
The scalar product $(m_{i,j})$ of two  bi-infinite normalised paths $(p_i/q_i)$
and $(r_j/s_j)$ in $\T$ forms a tame $\SL_2(\Z[\sigma])$-tiling.

\end{prop}

\begin{proof}
 A direct calculation shows that the scalar product forms an  $\SL_2(\Z[\sigma])$-tiling (cf. the proof of Proposition~\ref{det-l}).
  The tameness follows from the fact that $(m_{i,j})$ is a product of two (infinite) matrices of rank two.

\end{proof}

\begin{remark}
Note that tameness is equivalent to the requirement that every row of matrix $(m_{i,j})$ is a linear combination of two adjacent rows, say $(m_{0,j})$ and $(m_{1,j})$.

\end{remark}

\begin{example}
  \label{ex}
  Consider two periodic paths in ${\mathcal T}$: $(u_i)=(\dots, \infty,0,1,\sigma,\infty,\dots)$, and $(v_j)=(\dots,0,\sigma^2,-1,\infty, 0,\dots)$. Normalise the paths as follows:
  $$u_0=\frac10,\ u_1=\frac01,\ u_2=\frac{-1}{-1},\ u_3=\frac{\bar \sigma}{-\sigma},\ \text{and}\quad u_{i+4}=\frac{\sigma^{{(-1)}^{i+1}}}{\sigma^{{(-1)}^{i+1}}}u_i;$$
  $$v_0=\frac0{-\bar\sigma},\ v_1=\frac{\sigma}{\bar\sigma},\ v_2=\frac{-1}{1},\ v_3=\frac{-1}{0},\ \text{and}\quad v_{i+4}=\frac{\sigma^{{(-1)}^{i}}}{\sigma^{{(-1)}^{i}}}v_i.$$
  The initial four values and the recursive formulae define the (bi-infinite) normalised paths uniquely. By taking the scalar product, we obtain a tiling shown in Fig.~\ref{tiling-ex}. The tiling has a block structure with the initial block of size $4\times 4$ shown in Fig.~\ref{tiling-ex}, with
  $$m_{i+4k,j+4l}=\sigma^{l-k}m_{i,j}.$$
  Note that we get a counterpart of an antiperiodic $SL_2$-tiling considered in~\cite[Section 3.3]{MGOT}. The reason for this is the periodicity of the paths $(u_i)$ and $(v_j)$.

  \end{example}

  \begin{figure}[!h]
    \begin{center}
      $$      
      \begin{array}{c|ccccccc}
    & \ldots&\frac{0}{-\bar\sigma}& \frac{\sigma}{\bar\sigma} &\frac{-1}{1} & \frac{-1}{0} & \frac{0}{-1}& \ldots \phantom{\int\limits_{A_A}^{A}}\\
       \hline
    \vdots & \ddots&&& \vdots&&& \iddots\\    
\frac{1}{0}     &&\cellcolor{blue!7} 0&\cellcolor{blue!7} \sigma &\cellcolor{blue!7} -1 &\cellcolor{blue!7} -1 & 0&    \phantom{\int\limits^{-}}\\
 \frac{0}{1}      & &\cellcolor{blue!7}-\bar\sigma &\cellcolor{blue!7} \bar \sigma&\cellcolor{blue!7} 1&\cellcolor{blue!7} 0& -1&                \phantom{\int\limits^{A}}\\
 \frac{-1}{-1}    & \ldots &\cellcolor{blue!7}\bar\sigma &\cellcolor{blue!7}-1 &\cellcolor{blue!7} 0 &\cellcolor{blue!7} 1& 1& \ldots            \phantom{\int\limits^{A}}\\
 \frac{\bar\sigma}{-\sigma}     & &\cellcolor{blue!7}1 &\cellcolor{blue!7} 0 &\cellcolor{blue!7} -1  &\cellcolor{blue!7} -\bar\sigma & \sigma&    \phantom{\int\limits^{A}}\\
\frac{\bar\sigma}{0}     & &0& 1& -\bar\sigma & -\bar\sigma & 0&          \phantom{\int\limits^{A}}\\
   \vdots  &  \iddots& && \vdots&&&\ddots\\    
    \end{array}
    $$   
\caption{$SL_2(\Z[\sigma]$)-tiling  as a scalar product of two normalised paths in $\mathcal T$. The shaded region is the initial block, all shifted blocks are obtained from that one by multiplication by powers of $\sigma$ (see Example~\ref{ex})}
\label{tiling-ex}
\end{center}
\end{figure}

\begin{prop}\label{sl-scalar}
Any tame $\SL_2(\Z[\sigma])$-tiling is a scalar product of two normalised paths.
\end{prop}

\begin{proof}
  Let $(m_{ij})$ be a  tame $\SL_2(\Z[\sigma])$-tiling.
  First, we construct the path $(r_j/s_j)$.
  Set $(p_0/q_0,p_1/q_1)=(1/0,0/1)$. Then we can set $r_j=m_{0,j}, s_j=m_{1,j}$. This would imply that condition
  $m_{i,j}=p_ir_j+q_is_j$ is satisfied for $i=0$ and all $j$.  Notice that  since
  $$\det \begin{pmatrix}m_{0,j}&m_{0,j+1}\\m_{1,j}&m_{1,j+1}\\ \end{pmatrix}=1,$$
 each fraction $r_j/s_j=m_{0,j}/m_{1,j}$ is irreducible, and by the same reason the path  $(r_j/s_j)$ is normalised.

 Now, we use the path  $(r_j/s_j)$ and the values  $(p_0/q_0,p_1/q_1)=(1/0,0/1)$ to construct $(p_i/q_i)$ for all $i$.
 We know from tameness of  $(m_{i,j})$ that every row of $(m_{i,j})$ is a linear combination of the rows $(m_{0,j})$ and $(m_{1,j})$. Denote by $\alpha_i$ and $\beta_i$ the coefficients of this linear combination for the $i$-th row, i.e.
 $$m_{i,j}=\alpha_i m_{0,j}+\beta_i m_{1,j}.$$
 The values $\alpha_i$ and $\beta_i$ are uniquely defined from  
 \begin{equation}\label{alpha-beta}
 \begin{pmatrix} \alpha_i & \beta_i\\\end{pmatrix}
 \begin{pmatrix} m_{0,0}&m_{0,1}\\ m_{1,0}&m_{1,1}\end{pmatrix}=
 \begin{pmatrix} m_{i,0}&m_{i,1}\end{pmatrix},
 \end{equation}
 moreover,  $\alpha_i,\beta_i\in \Z[\sigma]$ as $\det \begin{pmatrix} m_{0,0}&m_{0,1}\\ m_{1,0}&m_{1,1}\end{pmatrix}=1$.
We set $p_i=\alpha_i$ and $q_i=\beta_i$. As   $r_j=m_{0,j}$ and $s_j=m_{1,j}$, this implies that $m_{i,j}=p_im_{0,j}+q_im_{1,j}=p_ir_j+q_is_j$ for all $i,j$.

We are left to check that the constructed path $(p_i/q_i)$ is normalised.
Notice that (\ref{alpha-beta}) implies that $\alpha_i/\beta_i$ is an irreducible fraction.
Also, notice that
$$
 \begin{pmatrix} \alpha_i&\beta_i\\ \alpha_{i+1} & \beta_{i+1}\\\end{pmatrix}
 \begin{pmatrix} m_{0,0}&m_{0,1}\\ m_{1,0}&m_{1,1}\end{pmatrix}=
 \begin{pmatrix}  m_{i,0}&m_{i,1}\\m_{i+1,0}&m_{i+1,1}\end{pmatrix}
$$
implies that $\det \begin{pmatrix} \alpha_i&\beta_i\\ \alpha_{i+1} & \beta_{i+1}\end{pmatrix}=1$, and hence, the path is normalised.

\end{proof}

\begin{remark}
\label{unique for 10 01}
Notice that once the initial values  $(p_0/q_0,p_1/q_1)=(1/0,0/1)$ are fixed, the sequences $p_i/q_i$ and $s_j/r_j$ are uniquely determined by $(m_{i,j})$. 

\end{remark}

\begin{remark}\label{sl2-normalised}
An easy computation shows that the action of $\SL_2(\Z[\sigma])$ on paths  takes normalised paths to normalised ones (cf. the proof of Proposition~\ref{det-l}).

\end{remark}  

\begin{lemma}\label{isom->equiv}
  Let $p_i/q_i$ and $r_j/s_j$ be two normalised paths. Let $A\in \SL_2(\Z[\sigma])$  and consider
  $$ \begin{pmatrix} p_i'\\ q_i'\end{pmatrix} =A \begin{pmatrix} p_i\\ q_i\end{pmatrix},  \qquad
     \begin{pmatrix} r_j'\\ s_j'\end{pmatrix} =(A^T)^{-1} \begin{pmatrix} r_j\\ s_j\end{pmatrix}.
     $$
     Let $M$ and $M'$ be  $\SL_2(\Z[\sigma])$-tilings given by
     $$M= ((p_i/q_i)\cdot (r_j/s_j)) \quad \text{and} \quad  M'= ((p_i'/q_i')\cdot (r_j'/s_j')).$$
     Then  $M$ and $M'$ coincide.

\end{lemma}

\begin{proof} This is a direct computation:
  \begin{multline*}
(m_{ij}')= \begin{pmatrix} ...&...\\  p_{i-1}' &q_{i-1}' \\ p_i' & q_i' \\ p_{i+1}'& q_{i+1}'\\ ...&....\end{pmatrix}
  \begin{pmatrix} ...& r_{j-1}'&r_j'&r_{j+1}'&... \\...& s_{j-1}'&s_j'&s_{j+1}'&...  \end{pmatrix} = \\
   \begin{pmatrix} ...&...\\  p_{i-1} &q_{i-1} \\ p_i & q_i \\ p_{i+1}& q_{i+1}\\ ...&....\end{pmatrix} A^T (A^T)^{-1}
   \begin{pmatrix} ...& r_{j-1}&r_j&r_{j+1}&... \\...& s_{j-1}&s_j&s_{j+1}&...  \end{pmatrix}
      =(m_{ij}).
  \end{multline*}
\end{proof}

\begin{remark}\label{rem18}
  Notice that different normalisations of the paths may lead to different $\SL_2(\Z[\sigma])$-tilings. More precisely,
if $p_i$ and $q_i$ are multiplied by $\sigma^{(-1)^i k}$ (cf. Remark~\ref{18-similarity}), then $i$-th row of  $(m_{i,j})$ is multiplied by  $\sigma^{(-1)^i k}$.
Similarly, changing the normalisation of $(r_j/s_j)$ affects   columns of  $(m_{i,j})$.
This leads to $36$ tame $\SL_2(\Z[\sigma])$-tilings, generically $18$ of them are distinct tilings (as a simultaneous multiplication of  columns and rows by $-1$ preserves the tiling).
\end{remark}

\begin{definition}\label{def-equiv}
We call $\SL_2(\Z[\sigma])$-tilings obtained from different normalisations of the same pair of paths {\it equivalent}.

\end{definition}

\begin{lemma}\label{inj}
  Let  $((p_i/q_i),(r_j/s_j))$ and  $((p_i'/q_i'),(r_j'/s_j'))$ be two pairs of normalised paths and
  $M=(p_i/q_i)\cdot(r_j/s_j)$, $M'=(p_i'/q_i')\cdot(r_j'/s_j')$ be their scalar products.

  If $M$ is equivalent to $M'$ then there exists  a matrix  $A\in \SL_2(\Z[\sigma])$ and  normalisations   $ p_i''/ q_i''$  and  $\ r_j''/ s_j''$ of  $p_i'/q_i'$ and    $r_j'/s_j'$ respectively  such that
  $$ \begin{pmatrix}  p_i''\\  q_i''\end{pmatrix} =A \begin{pmatrix} p_i\\ q_i\end{pmatrix},  \qquad
     \begin{pmatrix}  r_j''\\  s_j''\end{pmatrix} =(A^T)^{-1} \begin{pmatrix} r_j\\ s_j\end{pmatrix}.
     $$

\end{lemma}

\begin{proof}
First, as $M=(p_i/q_i)\cdot (r_j/s_j)$ is equivalent to $M'=(p_i'/q_i')\cdot (r_j'/s_j')$, there exists  normalisations
  $p_i''/ q_i''$  and  $ r_j''/ s_j''$ of  $p_i'/q_i'$ and    $r_j'/s_j'$  such that $M=( p_i''/q_i'')\cdot ( r_j''/ s_j'')$.

  Next, let $X,Y\in \SL_2(\Z[\sigma])$ be matrices such that
  $$ X\begin{pmatrix} p_0 & p_1 \\ q_0 & q_1 \end{pmatrix} =\begin{pmatrix}1&0 \\ 0& 1 \end{pmatrix},  \qquad
 Y\begin{pmatrix} p_0'' & p_1'' \\ q_0'' &  q_1'' \end{pmatrix} =\begin{pmatrix}1&0 \\ 0& 1 \end{pmatrix}.
 $$
 Consider the following two pairs of sequences $\big((\bar p_i/\bar q_i), (\bar r_j/\bar s_j)\big)$ and
 $\big(({\bar p_i}''/{\bar  q_i}''), ({\bar  r_j}''/{\bar s_j}'')\big)$ given by:
$$
\begin{pmatrix} \bar p_i \\ \bar q_i \end{pmatrix}= X \begin{pmatrix}  p_i \\  q_i \end{pmatrix}, \qquad
\begin{pmatrix} \bar r_j \\ \bar s_j \end{pmatrix}= (X^T)^{-1} \begin{pmatrix}  r_j \\  s_j \end{pmatrix}
$$
and
$$
\begin{pmatrix} {\bar p_i}'' \\ {\bar q_i}'' \end{pmatrix}= Y \begin{pmatrix}  p_i'' \\  q_i'' \end{pmatrix}, \qquad
\begin{pmatrix} {\bar r_j}'' \\ {\bar s_j}'' \end{pmatrix}= (Y^T)^{-1} \begin{pmatrix}  r_j'' \\  s_j'' \end{pmatrix}.
$$
By Remark~\ref{unique for 10 01}  these two pairs of sequences coincide, so we get the condition of the lemma
satisfied for $A=Y^{-1}X$.

\end{proof}

\begin{theorem}\label{sigma-sh}
The scalar product provides  a bijection between equivalence classes of the tame  $\SL_2(\Z[\sigma])$-tilings  and
pairs of paths in $\T$ considered up to simultaneous action of $\SL_2(\Z[\sigma])$ by $A\in \SL_2(\Z[\sigma])$ on one of the paths and $(A^T)^{-1}$ on the other.

\end{theorem}

\begin{proof}
 By Proposition~\ref{paths->tiling} the scalar product maps any pair of normalised paths in $\T$ to a  tame  $\SL_2(\Z[\sigma])$-tiling. We will prove that this map provides the required bijection.

First, notice that by Definition~\ref{def-equiv} different normalisations of the paths lead to  equivalent   $\SL_2(\Z[\sigma])$-tilings.
Next, by Lemma~\ref{isom->equiv} pairs of paths equivalent under the action of  $\SL_2(\Z[\sigma])$ result in the same  $\SL_2(\Z[\sigma])$-tiling. This implies that the scalar product is a well-defined map from
pairs of paths in $\T$  (up to simultaneous action of $\SL_2(\Z[\sigma])$) to equivalence classes of  $\SL_2(\Z[\sigma])$-tilings. This map is surjective by Proposition~\ref{sl-scalar} and injective by Lemma~\ref{inj}.

\end{proof}

\begin{remark}
\label{det-scal}
  In~\cite{Sh}, a similar bijection was constructed for $SL_2(\Z)$-tilings using determinant instead of the scalar product, i.e. using the map $((p_i/q_i),(r_j/s_j)) \mapsto (m_{i,j})$ given by
$$
m_{i,j}=p_is_j-q_ir_j.
$$
Our construction can also be formulated in these terms. Namely,  replace the path $(r_j,s_j)$ by the path $(s_j,-r_j)$. Then the scalar product is exactly replaced by computing the determinant. Notice that the transformation  $(r_j,s_j) \to (-s_j,r_j)$ is given  by the map
$z\to -1/z$, or, in other words, by the action of  $\begin{pmatrix} 0&-1\\ 1&0\end{pmatrix}\in \SL_2(\Z)$ applied to the path   $(r_j,s_j)$.

Note that the construction with determinant is also invariant under isometries, i.e. given a pair of normalised paths   $((p_i/q_i),(r_j/s_j))$ and   a matrix $A\in  \SL_2(\Z[\sigma])$, the paths  $(A(p_i/q_i),A(r_j/s_j))$ obtained by simultaneous action of $A$ on both paths define the same
 $\SL_2(\Z[\sigma])$-tiling as  $((p_i/q_i),(r_j,s_j))$ (cf. Lemma~\ref{isom->equiv}).

 \smallskip

Furthermore, given standard horospheres at all points of $\QQ(\sigma)$,
 the absolute value of the element $m_{i,j}=p_is_j-q_ir_j$ of the constructed  $\SL_2(\Z[\sigma])$-tiling is the $\lambda$-length of the arc connecting  $p_i/q_i$ with $r_j/s_j$ (see Theorem~\ref{lambda=l}).

\end{remark}

In view of Remark~\ref{det-scal}, the result of  Theorem~\ref{sigma-sh} can be reformulated as follows.

\begin{theorem}\label{sh-det}
The map   $((p_i/q_i),(r_j/s_j)) \mapsto (m_{i,j}=p_is_j-q_ir_j)$ provides a bijection between equivalence classes of  tame  $\SL_2(\Z[\sigma])$-tilings  and
pairs of paths in $\T$ considered up to simultaneous action of $\SL_2(\Z[\sigma])$  on both paths.

\end{theorem}

\begin{example}
  \label{ex-c1}
  The $SL_2(\Z[\sigma])$-tiling constructed in Example~\ref{ex} can be obtained using the procedure above from periodic paths $(u_i)=(\dots, \infty,0,1,\sigma,\infty,\dots)$, and $(v_j')=(\dots, \infty, \sigma, 1,0,\infty,\dots)$.

  \end{example}

\begin{definition}
  We say that two paths $(p_i,q_i)$ and $(r_j,s_j)$ are {\it coplanar in $\T$} if there exists a hyperbolic plane $\Pi$ containing a face of some fundamental tetrahedron in $\T$ and such that $p_i/q_i, r_j/s_j\in \Pi$ for all $i,j$.

\end{definition}

The following is an immediate corollary of Theorem~\ref{sh-det}.

\begin{cor}\label{real}
  Let  $(p_i,q_i)$ and $(r_j,s_j)$ be two normalised paths on  $\mathcal T$, and let  $(m_{i,j})=(p_is_j-q_ir_j)$ be the $\SL_2(\Z[\sigma])$-tiling defined by these paths. Then the following are equivalent:
  \begin{enumerate}
  \item[(a)] paths  $(p_i,q_i)$ and $(r_j,s_j)$  are coplanar;
  \item[(b)] $m_{i,j}\in \Z$ for all $i,j\in\Z$;
\item[(c)] there are two consecutive  rows $k,k+1$ and two consecutive columns $l,l+1$, such that $m_{i,l}, m_{i,l+1}\in \Z$, and  $m_{k,j},m_{k+1,j}\in \Z$ for all $i,j\in \Z$.

\end{enumerate}

\end{cor}

\begin{proof}
\begin{itemize}
\item``(a)$\Rightarrow$(b)'': Given two paths lying in a plane $\Pi$, there exists  $A\in \SL_2(\Z[\sigma])$ taking $\Pi$ to the vertical plane containing the real line. Then all the points $A(p_i/q_i)$ and $A(r_j/s_j)$ are real, so the resulting paths lie in a copy of the Farey graph contained in $A(\Pi)$, which implies (b).

\item``(b)$\Rightarrow$(c)'': This is a trivial implication.

\item``(c)$\Rightarrow$(a)'': We show that condition (c) implies that both paths  $(p_i/q_i)$ and $(r_j/s_j)$ are real rational numbers (and hence (a) follows). Showing this is equivalent to proving that both  $(p_i/q_i)$ and $(-s_j/r_j)$ are real. In terms of the latter pair of paths, the $SL_2(\Z[\sigma])$-tiling $(m_{i,j})$ is the scalar product of paths. 
Now,  we apply the construction from Proposition~\ref{sl-scalar}, and get that $p_i/q_i$ and $-s_j/r_j$ are real rationals.
\end{itemize}
\end{proof}

\section{Paths in $\mathcal T$ and sequences in $\Z[\sigma]$}
\label{sec-path}

In this section we show that paths in $\mathcal T$ can be parameterised by (bi-infinite) sequences of elements of $\Z[\sigma]$. This relates the results of the previous section to classification of $SL_2$-tilings obtained in~\cite{BR}.

\begin{definition}
A path $(v_i)_{i=m}^{n}$ in $\mathcal T$  represented by irreducible fractions $v_i=p_i/q_i$ for all $i$ is called  {\it skew-normalised}
if 
$$
\det
\begin{pmatrix} p_i& p_{i+1} \\ q_i&q_{i+1}
\end{pmatrix}
=(-1)^i
$$
for all  admissible $i$.
\end{definition}

Similarly to the case of normalised paths, the property of being skew-normalised is preserved under the action of $\SL_2(\Z[\sigma])$.

\begin{definition}
 \label{sin_T}
 Consider a skew-normalised path $(v_{i-1},v_{i},v_{i+1})$ in $\mathcal T$
with $v_i=p_i/q_i$ be irreducible fractions in $\Z[\sigma]$.
Assume that
$$
\begin{pmatrix} p_{i+1} \\ q_{i+1}
\end{pmatrix}=
\begin{pmatrix} p_{i-1}& p_{i} \\ q_{i-1}&q_{i}
\end{pmatrix} \cdot
\begin{pmatrix} 1 \\ a
\end{pmatrix}.
$$
We say that $a$ is the {\it oriented $\T$-angle of $v_{i-1}v_{i}v_{i+1}$} and denote it by $\tsin (v_{i-1}v_{i}v_{i+1})$.
\end{definition}

\begin{prop}
\label{a}
\begin{itemize}
\item[(a)] Given a skew-normalised path $(v_{i-1},v_{i},v_{i+1})$  in $\mathcal T$,
  there exists a unique $a\in\Z[\sigma]$
such that
$
\tsin (v_{i-1}v_{i}v_{i+1})=a.
$
\item[(b)]
For every skew-normalised path $(v_{i-1},v_{i})$ in $\mathcal T$ and every $a\in\Z[\sigma]$ there exists
a unique  $v_{i+1}\in \mathcal T_0$  such that $(v_{i-1},v_{i},v_{i+1})$ is a skew-normalised path and
$
\tsin (v_{i-1}v_{i}v_{i+1})=a.
$

\item[(c)]
For every skew-normalised path $(v_{i},v_{i+1})$ in $\mathcal T$ and every $a\in\Z[\sigma]$ there exists
a unique  $v_{i-1}\in \mathcal T_0$  such that $(v_{i-1},v_{i},v_{i+1})$ is a skew-normalised path and
$
\tsin (v_{i-1}v_{i}v_{i+1})=a.
$

\end{itemize}
\end{prop}

\begin{proof}
To prove~(a),  note that
\begin{equation}
\label{aa}  
\begin{pmatrix} p_{i}&p_{i+1} \\ q_{i}&q_{i+1}
\end{pmatrix}=
\begin{pmatrix} p_{i-1}& p_{i} \\ q_{i-1}&q_{i}
\end{pmatrix} \cdot
\begin{pmatrix} 0&\xi \\ 1&\alpha
\end{pmatrix}
\end{equation}
for some $\alpha,\xi \in \C$ which are defined uniquely.

Since the path  $(v_{i-1},v_{i},v_{i+1})$ is skew-normalised,
$$
\det
\begin{pmatrix} p_{i-1}& p_{i} \\ q_{i-1}&q_{i}
\end{pmatrix}=-1, \qquad
\det\begin{pmatrix} p_{i}&p_{i+1} \\ q_{i}&q_{i+1}
\end{pmatrix}=1,$$
which implies that $\xi=1$. Equation~(\ref{aa}) can now be rewritten as
$$
\begin{pmatrix} p_{i-1}&p_{i+1} \\ q_{i-1}&q_{i+1}
\end{pmatrix}=
\begin{pmatrix} p_{i-1}& p_{i} \\ q_{i-1}&q_{i}
\end{pmatrix} \cdot
\begin{pmatrix} 1&1 \\ 0&\alpha
\end{pmatrix}.
$$

Therefore, we conclude 
$$ \tsin  (v_{i-1}v_{i}v_{i+1})=\alpha=
\frac{\det\begin{pmatrix}  p_{i-1}& p_{i+1} \\ q_{i-1}&q_{i+1} \end{pmatrix}}{\det\begin{pmatrix}  p_{i-1}& p_{i} \\ q_{i-1}&q_{i} \end{pmatrix}}= (-1)^{i-1}\det\begin{pmatrix}  p_{i-1}& p_{i+1} \\ q_{i-1}&q_{i+1} \end{pmatrix} ,
$$
which proves (a).

Parts (b) and (c) now follow from~(\ref{aa}) with $\xi=1$ and $\alpha=a$.

\end{proof}

\begin{definition}
Given a skew-normalised path  $(v_i)$,  define $a_i=\tsin  (v_{i-1}v_{i}v_{i+1})$. We call the sequence $(a_i)$  the {\it $\mathcal T$-angle sequence} of the path $(v_i)$.

\end{definition}

\begin{remark}
 $\mathcal T$-angle  sequence is a counterpart of  {\it itinerary} appearing in~\cite{Sh} and {\it  quiddity sequence} appearing in the context of friezes~\cite{CC}. This sequence (up to change of certain signs) also appears in~\cite[Section 3]{BR}.

\end{remark}

\begin{prop}\label{a->path}
\begin{itemize}  
\item[(a)] Every bi-infinite sequence $(a_i)$, $a_i\in \Z[\sigma]$, is a $\mathcal T$-angle sequence of some skew-normalised path.
\item[(b)] 
  Two skew-normalised paths $(p_i/q_i)$ and $(r_j/s_j)$ have the same $\mathcal T$-angle sequence if and only if
the paths are equivalent under $SL_2(\Z[\sigma])$-action.  
  
\end{itemize}  
\end{prop}

\begin{proof}
  (a) follows from parts (b) and (c) of Proposition~\ref{a}.
  
To prove (b), note that invariance of  $\mathcal T$-angle sequence under the action of  $SL_2(\Z[\sigma])$
 follows immediately from Definition~\ref{sin_T}: left multiplication by the same matrix does not affect the validity of 
 the equation.

 Conversely, consider a skew-normalised  path $(v_i)$.  Parts (b) and (c) of Proposition~\ref{a} imply that the fractions $v_0=p_0/q_0$ and $v_1=p_1/q_1$ together with the sequence $(a_i)$ uniquely define the whole  path $(v_i)$. 
Note that there exists  $A\in SL_2(\Z[\sigma])$ taking $(v_0,v_1)$ to
 $(\frac{1}{0},\frac{0}{1})$. Therefore, given two skew-normalised paths $(v_i)$ and $(v_i')$, there exists an element of $SL_2(\Z[\sigma])$ taking  $(v_i)$ to $(v_i')$, which completes the proof.

\end{proof}  

\begin{remark}
\label{sim}  
Given a path $(v_i)$ in $\mathcal T$, there exist $6$ its skew-normalisations obtained from one of them by multiplying both numerator and denominator of $v_i$ by $\sigma^{(-1)^i\cdot k}$  (cf. Remark~\ref{18-similarity}). 
 $\mathcal T$-angle sequence of the new path is then obtained by multiplying  $a_i$ by   $\sigma^{(-1)^{i+1}\cdot 2k}$.

\end{remark}

\begin{definition}
  Consider two sequences $(a_i)$ and $(b_i)$ where $a_i,b_i\in \Z[\sigma]$. We say that $(a_i)$ and $(b_i)$ are   {\it equivalent}
 if there exists  $k\in \{-1,1\}$ such that  $b_i=a_i \sigma^{(-1)^{i+1}\cdot 2k}$.

\end{definition}  

An equivalence class of sequences corresponds to a path in $\mathcal T$ considered up to the action of $PSL_2(\Z[\sigma])$.

We introduce one more notion with the aim to  reformulate Theorem~\ref{sh-det} in terms of  $\mathcal T$-angle sequences. Namely, the paths from   Theorem~\ref{sh-det} are replaced by their $\mathcal T$-angle sequences,  and the relative position of the paths with respect to each other is indicated by a matrix in $SL_2(\Z[\sigma])$.

\begin{definition}
  Consider a triple $\big( (a_i),(b_j),X\big)$, where $(a_i)$ and $(b_j)$ are equivalence cla\-sses of sequences,   $a_i,b_j\in \Z[\sigma]$, and $X\!\in SL_2(\Z[\sigma])$. We say that two triples  $\big( (a_i),(b_j),X\big)$ and  $\big( (\tilde a_i),(\tilde b_j),\widetilde X\big)$ are {\it equivalent} if the following two conditions hold:
\begin{itemize}
\item[(1)]  $(a_i)$ is equivalent to $(\tilde a_i)$, and $(b_i)$ is equivalent to $(\tilde b_i)$;
\item[(2)]  if $\tilde a_0=a_0 \sigma^{-2k}$ and $\tilde b_0=b_0 \sigma^{-2l}$, then
  $\widetilde X= \pm \begin{pmatrix}\sigma^l&0\\0 &\sigma^{-l}    \end{pmatrix}X \begin{pmatrix}\sigma^{-k}&0\\0 &\sigma^{k}    \end{pmatrix}$.  

\end{itemize}
\end{definition}  

\begin{remark}\label{18}
  Substituting $X$ with $-X$ corresponds to changing the sign of  all numerators and denominators in one of the normalised paths. In particular, generic equivalence class of triples consists of 18 elements. 
  
\end{remark}  

\begin{theorem}\label{classes}
There exists a bijection between equivalence classes of  tame  $\SL_2(\Z[\sigma])$-tilings  and equivalence classes of
triples $\big( (a_i),(b_j),X\big)$, where $(a_i)$ and $(b_j)$ are  sequences,   $a_i,b_j\in \Z[\sigma]$, and $X\in SL_2(\Z[\sigma])$. 

\end{theorem}  

\begin{proof}
  Given a triple  $\big( (a_i),(b_j),X\big)$ we construct an $SL_2(\Z[\sigma])$-tiling as follows.
  According to Proposition~\ref{a->path}, there exists a unique skew-normalised path $(u_i)$   with $\mathcal T$-angle sequence $(a_i)$ and $u_0=p_0/q_0$, $u_1=p_1/q_1$, where
$$
\frac{p_0}{q_0}=\begin{cases} 1/0 & \text{ if $0\le \arg(a_0)<2\pi/3 $},\\
                              \sigma^{-1}/0 & \text{ if $2\pi/3\le \arg(a_0)<4\pi/3 $},\\
                              \sigma^{-2}/0 & \text{ if $4\pi/3\le \arg(a_0)<2\pi $};\\
                            \end{cases}
\qquad \quad
\frac{p_1}{q_1}=\frac{0}{p_0^{-1}}.                            
$$
Similarly,  there exists a unique skew-normalised path $(v_j)$ with $\mathcal T$-angle sequence $(b_j)$ and $v_0=r_0/s_0$, $v_1=r_1/s_1$, where
$$
\begin{pmatrix} r_0 &r_1 \\ s_0&s_1
\end{pmatrix}  
=\begin{cases} \begin{pmatrix} x_{1,1} & x_{1,2}\\x_{2,1}&x_{2,2}    \end{pmatrix}     & \text{ if $0\le \arg(b_0)<2\pi/3 $},\\
         \begin{pmatrix}  \sigma^{-1}x_{1,1} &  \sigma x_{1,2}\\ \sigma^{-1}x_{2,1}&  \sigma x_{2,2}    \end{pmatrix}  & \text{ if $2\pi/3\le \arg(b_0)<4\pi/3 $},\\
       \begin{pmatrix}  \sigma^{-2}x_{1,1} &  \sigma^2 x_{1,2}\\ \sigma^{-2}x_{2,1}&  \sigma^2 x_{2,2}    \end{pmatrix}                         & \text{ if $4\pi/3\le \arg(b_0)<2\pi $}.\\
                            \end{cases}
$$

 We can now construct a normalised path $(u_i')$ according to the following rule: 
  if $u_i=p_i/q_i$ then
  $$
  u_i'=\begin{cases}
\frac{p_i}{q_i} & \text{ if $i\equiv 0,1 \pmod 4   $} \\
\frac{-p_i}{-q_i} & \text{  if  $i\equiv 2,3 \pmod 4  $ }\\
    \end{cases}
  $$
Similarly, we use the skew-normalised path $(v_j)$ to construct   a normalised path $(v_j')$.
Now, a tame $SL_2(\Z[\sigma])$-tiling is defined by $m_{i,j}=p_i's_j'-q_i'r_j'$, where $u_i'=p_i'/q_i', v_j'=r_j'/s_j'$ (see Remark~\ref{det-scal}).

If sequences $(\tilde a_i)$ and $(a_i)$ are equivalent, then,  by Remark~\ref{sim}, the corresponding skew-normalised paths $(\tilde u_i)$ and $(u_i)$ are distinct skew-normalisations of the same path in $\mathcal T$. Hence, $(\tilde u_i')$
and $(u_i')$ are distinct normalisations of the same path in $\mathcal T$. Since $X$ is substituted with $\widetilde X$, the sequence $(b_j)$ still leads to the same path $(v_j)$  in $\mathcal T$. Therefore, by Theorem~\ref{sh-det}, 
the triple  $\big( (\tilde a_i),(\tilde b_j),\widetilde X\big)$ leads to an equivalent  $SL_2(Z[\sigma])$-tiling. 
Similarly, changing $(b_j)$ (and $X$) also leads  an equivalent  
$SL_2(Z[\sigma])$-tiling.

 Therefore, we have constructed a well-defined map from  equivalence classes of triples  $\big( (a_i),(b_j),X\big)$  to  equivalence classes of $SL_2(\Z[\sigma])$-tilings.

 To prove surjectivity of this map, consider any two paths in $\mathcal T$ and the corresponding class of  $SL_2(\Z[\sigma])$-tilings. Choose any skew-normalisations $(u_i)$ and $(v_j)$ of these paths and set $(a_i)$ and $(b_j)$ to be the corresponding $\mathcal T$-angle sequences. Set $X$ to be the element of 
 $SL_2(\Z[\sigma])$ taking $u_0$ to $v_0$ and  $u_1$ to $v_1$. Then the triple  $\big( (a_i),(b_j),X\big)$  produces an  $SL_2(\Z[\sigma])$-tilings from the required equivalence class. This proves surjectivity. 

 Notice that taking different skew-normalisations of the paths in the construction above (and, hence, the corresponding matrix $X$) leads to an equivalent triple. This shows injectivity.

\end{proof}

The next theorem is a very particular case of Proposition~3 of~\cite{BR} where it is proved in purely linear-algebraic terms (cf. also~\cite[Theorem 2]{MGOT}).

\begin{theorem}
 \label{bijection} 
There exists a bijection between the  tame  $\SL_2(\Z[\sigma])$-tilings  and 
triples $\big( (a_i),(b_j),X\big)$, where $(a_i)$ and $(b_j)$ are  sequences,   $a_i,b_j\in \Z[\sigma]$, and $X\in SL_2(\Z[\sigma])$. 

\end{theorem}  

\begin{proof}
 The map constructed in the proof of  Theorem~\ref{classes}  takes equivalence classes of triples $\big( (a_i),(b_j),X\big)$ to equivalence classes of tame  $\SL_2(\Z[\sigma])$-tilings.
  Due to Proposition~\ref{sl-scalar}, the same map is a surjective map from  triples  $\big( (a_i),(b_j),X\big)$ to  
  tame  $\SL_2(\Z[\sigma])$-tilings. We now observe that generically  equivalence classes of triples and  equivalence classes of  $\SL_2(\Z[\sigma])$-tilings consist of 18 elements each (see Remarks~\ref{rem18} and~\ref{18}), which shows that the map is one-to-one in the generic case.
  
  We are left to deal with singular equivalence classes. The small classes of tilings appear when $m_{i,j}=0$ for all $i+j$ even (or all $i+j$ odd). In every such class there are precisely six tilings, and it is easy to see that  corresponding
$\mathcal T$-angle sequences $(a_i)$ and $(b_j)$ are both identically zero. The matrix $X$ in that case can take precisely six values, so that the equivalence class of triples is also of size 6. 
This completes the proof of the theorem.
  
\end{proof}  

\section{Further comments}
\label{comments}
\subsection{Skew-normalised paths and continued fractions}

Consider a finite skew-norma\-lised path  $(v_i)_{i=0}^{n+2}$ and let $v_0=1/0$ and $v_1=0/1$. Let $(a_i)_{i=1}^{n+1}$ be its $\mathcal T$-angle sequence. Then  $(a_i)_{i=1}^{n+1}$ provides a continued fraction expansion for $v_{n+2}$:
$$
v_{n+2}=[a_1; a_{2},\ldots , a_{n+1}]= a_1+\frac{1}{\displaystyle a_2+\frac{1}{\displaystyle
a_3+\frac{1}{\displaystyle\ddots+\frac{\displaystyle 1}{a_{n+1}}}}}.
$$
Indeed, as it was shown in the proof of Proposition~\ref{a},
$$
\begin{pmatrix} p_{i}&p_{i+1} \\ q_{i}&q_{i+1}
\end{pmatrix}=
\begin{pmatrix} p_{i-1}& p_{i} \\ q_{i-1}&q_{i}
\end{pmatrix} \cdot
\begin{pmatrix} 0&1 \\ 1&a_i
\end{pmatrix},
$$
and the statement becomes a classical property of continuants for  $[a_1; a_{2},\ldots , a_{n+1}]$.

\begin{cor}
A skew-normalised path $(u_i)$ visits the same vertex of $\mathcal T$ twice if and only if its $\mathcal T$-angle sequence $(a_i)$ contains a finite subsequence $(a_i)_{i=k}^{k+n}$, $k,n\in \Z$  such that $0=[a_k;a_{k+1},a_{k+2},\dots,a_{k+n}]$.

\end{cor}

\subsection{Some further properties of $\mathcal T$-angle } 
\label{prop T-angles}

Definition~\ref{sin_T} of $\mathcal T$-angle $\tsin (v_0v_1v_2)$ can be generalised to the situation when $v_0,v_1,v_2$ are any vertices of $\T$: here, we drop the requirement on $v_0v_1$ and $v_1v_2$ to be  edges of $\T$.
Namely, following the proof of Proposition~\ref{a}, one can define $a\in\Q(\sigma)$ such that
$$
\begin{pmatrix} p_{i+1} \\ q_{i+1}
\end{pmatrix}=
\begin{pmatrix} p_{i-1}& p_{i} \\ q_{i-1}&q_{i}
\end{pmatrix} \cdot
\begin{pmatrix} \xi \\ a
\end{pmatrix}
$$
for some $\xi\in\Q(\sigma)$ (notice that $a$ may not belong to $\Z[\sigma]$ anymore).

Next, we list some properties of $\mathcal T$-angles.

\begin{prop}
\label{identities}
Let $v_i=p_i/q_i$, where $p_i,q_i\in \Z[\sigma]$, $i=1,2,3$. Let $\det_{i,j}=p_iq_j-p_jq_i$. Then 
  \begin{itemize}
  \item[(1)]
     $\tsin (v_0v_1v_2)= \frac{\det_{0,2}}{\det_{0,1}}$;


\item[(2)]  $\tsin (v_0v_1v_2)\cdot\tsin (v_1v_2v_0)\cdot\tsin (v_2v_0v_1)=-1$.
  \end{itemize}
\end{prop}

\begin{proof}
  The first property is a result of a direct computation, the second follows from the first.

\end{proof}

There is also the following version of Ptolemy relation for $\T$-angles.

\begin{prop}
  Let $v_i=p_i/q_i$, where  $p_i,q_i \in \Z[\sigma]$, $i\in \{1,2,3,4\}$ be four distinct points lying on one circle or line in $\C$ in the cyclic order $v_1v_2v_3v_4$.  Let  $v_0=p_0/q_0$ be any other point in $\QQ(\sigma)$. Define
  $$x_{i,j}=\sqrt{|\tsin (v_iv_0v_j) \cdot \tsin (v_jv_0v_i)  |   }. $$
Then $x_{1,3}x_{2,4}=x_{1,2}x_{3,4}+x_{2,3}x_{1,4}.$

\end{prop}

\begin{proof}
  We prove the identity in assumption that all $p_i/q_i$ are irreducible. This does not affect  the validity of the statement as the identity is homogeneous.

   Denote $\det_{i,j}=p_iq_j-p_jq_i$.
  In view of Proposition~\ref{identities}(a),
  $$x_{i,j}=x_{j,i}=\sqrt{\big| \frac{\det_{i,j}\det_{j,i} }{\det_{i,0}\det_{j,0}}\big| } =\frac{|\det_{i,j}|}{\sqrt{|\det_{i,0}\det_{j,0}|}}.   $$
Theorem~\ref{lambda=l} implies that the determinants satisfy the Ptolemy relation, i.e.   
  \begin{center}
  \text{$ |\det_{1,3}||\det_{2,4}|=|\det_{1,2}||\det_{3,4}|+|\det_{2,3}||\det_{1,4}|$. }
\end{center}
Dividing every term of this equation by $\sqrt{|\det_{1,0}\det_{2,0}\det_{3,0}\det_{4,0}|}$, we obtain the required relation for $x_{i,j}$.

\end{proof}

\subsection{Other imaginary quadratic fields}

$\!\!$Let $K$ be an imaginary quadratic field $\Q(\!\sqrt{-d})$, where $d\in \Z_+$ is square-free, and let $\KK=K\cup \{\infty \}$.
Let $\mathcal O_K$ be the ring of integers of $K$, consider $K$ as the field of fractions of  $\mathcal O_K$.  
So, every point of $\KK$ can be written as an irreducible fraction $p/q$, where $p,q,\in   \mathcal O_K$. We consider $\KK$ as points of the boundary of $\H^3$, with Bianchi group $Bi(d)$ acting on $\KK$.

Following~\cite{Sch}, one can consider a graph $G$ with vertices at $\KK$ and edges between $p/q$ and $r/s$ whenever $|ps-rq |=1$. As for $d=3$, $G$ can be constructed as the orbit of a classical Farey graph $\mathcal F$ under the action of $Bi(d)$. Note that the graph is not always connected, see~\cite{St}. 

Triangular and quadrilateral faces of $G$ are described in~\cite{Sch} (up to the action of $Bi(d)$).

\subsubsection{Lambda lengths}
As for $d=3$, we can assign a standard horosphere at every point of  $\mathcal O_K$ (i.e., a sphere of Euclidean radius $1/2$), and then apply $SL_2(\mathcal O_K)$ to obtain a standard horosphere at every point of $\KK$. Two horospheres are tangent if and only if their centres are adjacent in $G$, and disjoint otherwise.

Proceeding as in Remark~\ref{radius} and using the results of~\cite{N}, we can deduce that the det-lengths still coincide with $\lambda$-lengths with respect to the standard horospheres.

In particular, this leads to the following corollary: given a normalised path $p_i/q_i$, the value $|p_{i}q_{i+2}-p_{i+2}q_{i}|$ coincides with the $\lambda$-length between $p_i/q_i$ and $p_{i+2}/q_{i+2}$.

\subsubsection{$SL_2$-tilings}
The results of Section~\ref{SL_2-tilings} are also valid for any imaginary quadratic field $K$. The only difference is that one needs to consider all units of $K$ rather than powers of $\sigma$ (for example, we need to do so in defining equivalent $SL_2(\mathcal O_K)$ tilings as in Definition~\ref{def-equiv}).

In particular, given two normalised paths in $G$, we can construct an $SL_2(\mathcal O_K)$-tiling by taking determinants of the pairs of entries. The construction provides a bijection between equivalence classes of tame $SL_2(\mathcal O_K)$-tilings and pairs of paths in $ G$ considered up to simultaneous action of $SL_2(\mathcal O_K)$ on both paths.

The counterpart of Theorem~\ref{bijection} for any $K$ is proved in~\cite{BR}, while the geometric approach via pairs of paths and their corresponding sequences of elements of $\mathcal O_K$ also works. In particular, modulus of $i$-th entry of the corresponding quiddity sequence will be equal to the $\lambda$-length between $(i-1)$-th and  $(i+1)$-th vertices of the path.

\subsubsection{Friezes}
\label{friezes}
Consider an $SL_2(\Z)$-tiling constructed via taking determinants (see Theorem~\ref{sh-det}), where two paths coincide. If the path is periodic with period $m>3$, we obtain a tame {\em frieze pattern} of height $m-3$ (see~\cite{MG}). It follows from~\cite{Sh} that all tame friezes can be obtained in this way. This can be generalized to $SL_2(\mathcal O_K)$-tilings for all imaginary quadratic fields using Schmidt arrangements described above: tame friezes of height $m-3$ without zeroes correspond precisely to non-self-intersecting closed (normalised) paths of length $m$.  

Recently, Cuntz and Holm~\cite{CH1,CH2} proved that the number of tame friezes of given  height without zeroes over $\mathcal O_K$ is finite for any $K=\Q(\sqrt{-d})$, and for $d\notin\{1,2,3,7,11\}$ all entries of friezes are actually rational integers. We would like to note that an independent proof of this result can be obtained by looking at the geometry of the corresponding graphs.

The finiteness of the number of friezes of given height without zeroes is equivalent to the finiteness of the number of closed non-self-intersecting paths up to the action of the Bianchi group. The latter is implied by the following observation: in a closed non-self-intersecting path of length $m$, the modulus of any entry of the quiddity sequence does not exceed $m-2$ (due to their relation to $\lambda$-lengths).

Further, it follows from~\cite[Theorem 5.1]{St} that for $d\notin\{1,2,3,7,11\}$ all vertices of any closed non-self-intersecting path in the corresponding graph belong to an image of $\QQ$ under an element of $Bi(d)$, which immediately implies that all entries of the frieze are integers.

In particular, all tame non-zero friezes over Eisenstein integers can be enumerated by closed paths in the tetrahedral graph (up to the symmetry group of the graph). 
It would be interesting to have a combinatorial way of a complete enumeration of all closed paths of a given length.

\end{document}